\documentclass[a4paper,british,12pt]{article}
\usepackage[utf8]{inputenc}
\usepackage{babel}
\usepackage{xcolor}
\usepackage{geometry}
\usepackage{verbatim}
\usepackage{mathtools}
\usepackage{amsmath}
\usepackage{amsthm}
\usepackage{amssymb}
\usepackage{mathdots}
\usepackage{stackrel}
\usepackage{esint}
\usepackage[unicode=true,pdfusetitle,
 bookmarks=true,bookmarksnumbered=false,bookmarksopen=false,
 breaklinks=false,pdfborder={0 0 1},backref=false,colorlinks=false]
 {hyperref}

\makeatletter

\pdfpageheight\paperheight
\pdfpagewidth\paperwidth

\usepackage{a4wide,color}
\newcommand{\red}[1]{\textcolor{black}{#1}}
\newcommand*{\id}{{\normalfont\hbox{1\kern-0.4em1}}}
\newcommand{\hide}[1]{}

\newcommand*{\curl}{\operatorname{curl}}
\newcommand*{\grad}{\operatorname{grad}}

\newcommand*{\dive}{\operatorname{div}}

\DeclareMathAccent{\Circ}{\mathalpha}{operators}{"17}
\newcommand{\interior}[1]{\Circ{#1}}

\newcommand{\lspan}{\operatorname{span}}

\newcommand{\oi}[2]{\left]#1,#2 \right[}

\newcommand{\lci}[2]{\left[#1,#2 \right[}

\renewcommand{\hat}{\widehat}
\renewcommand{\tilde}{\widetilde}
\renewcommand*{\epsilon}{\varepsilon}
\renewcommand*{\rho}{\varrho}
\theoremstyle{plain}
\newtheorem{thm}{Theorem}[section]
  \theoremstyle{definition}
  
  \newtheorem{hypo}[thm]{Assumption}
  \theoremstyle{remark}
  \newtheorem{rem}[thm]{Remark}
 \theoremstyle{definition}
  \newtheorem{example}[thm]{Example}
  \theoremstyle{plain}
  \newtheorem{prop}[thm]{Proposition}
  \theoremstyle{plain}
  \newtheorem{lem}[thm]{Lemma}
  \theoremstyle{plain}
  \newtheorem*{lem*}{Lemma}
  \theoremstyle{plain}
  \newtheorem*{thm*}{Theorem}
  \theoremstyle{plain}
  
  \theoremstyle{remark}
\allowdisplaybreaks

\makeatother

\begin{document}
\title{On a Class of Degenerate Abstract Parabolic Problems and Applications to Some Eddy Current Models}   

 \author{Dirk Pauly, Rainer Picard, Sascha Trostorff, and Marcus Waurick}
\maketitle 
\begin{abstract}We present an abstract framework for parabolic type equations which possibly degenerate on certain spatial regions. The degeneracies are such that the equations under investigation may admit a type change ranging from parabolic to elliptic type problems. The approach is an adaptation of the concept of so-called evolutionary equations in Hilbert spaces and is eventually applied to a degenerate eddy current type model. The functional analytic setting requires \textcolor{black}{quite}  minimal assumptions on the boundary and interface regularity. The degenerate eddy current model is justified as a limit model of non-degenerate hyperbolic models of Maxwell's equations.\end{abstract}  

\textbf{Classification} (MSC2010) 35Q61, 35M12, 35M32, 35K65, 35K90, 35L90

\section{\label{sec:A-Brief-Introduction}Introduction }

The dynamics of electromagnetic fields is described by Maxwell's equations,
which for classical materials take the form
\begin{align*}
\partial_{0}\epsilon\mathrm{E}+\sigma\mathrm{E}-\curl\mathrm{H} & =-\mathrm{J},\\
\partial_{0}\mu\mathrm{H}+\curl\mathrm{E} & =0,
\end{align*}
where $\partial_{0}$ denotes time-differentiation, $\mathrm{E}$
the electric field, $\mathrm{H}$ the magnetic field. The term $\mathrm{J}$
summarises external current densities exciting the field, $\epsilon$
and $\mu$ describe dielectricity and permeability of the medium,
$\sigma$ its conductivity. \textcolor{black}{Here, we consider Maxwell's equations subject to the electric boundary condition; that is, we ask the electric field to have a vanishing tangential component at the boundary of the underlying domain $\Omega \subseteq \mathbb{R}^3$ is regular enough to allow for a well-defined unit outward normal field $n\colon \partial\Omega \to \mathbb{R}^3$ the strong form of the mentioned boundary condition reads
\[
    E\times n =0 \text{ on }\partial\Omega.
\]
It is possible to generalise this condition in a similar way to homogeneous Dirichlet boundary conditions also to $\Omega$ lacking the regularity for a well-defined unit outward normal. This will be detailed later in the text. For the time being we shall use $\interior{\curl}$ to denote the $\curl$ operator with the additional constraint of the appropriate generalisation of vanishing tangential component at the boundary.
Consequently, the above mentioned Maxwell's equations subject to the boundary condition read
\[
\left(\partial_{0}\left(\begin{array}{cc}
\epsilon & 0\\
0 & \mu
\end{array}\right)+\left(\begin{array}{cc}
\sigma & 0\\
0 & 0
\end{array}\right)+\left(\begin{array}{cc}
0 & -\curl\\
\interior{\curl} & 0
\end{array}\right)\right)\left(\begin{array}{c}
\mathrm{E}\\
\mathrm{H}
\end{array}\right)=\left(\begin{array}{c}
-\mathrm{J}\\
0
\end{array}\right).
\]}
There is a suitable abstract framework, see \cite{Pi2009-1}, extended,
for example in \cite{PIC_2010:1889,PTW_Evo,Trost12,Trost18,Wau15,Wau16},
to incorporate dissipative, non-autonomous, and nonlinear systems.
If for example $\epsilon,\:\mu,\:\sigma$ are all selfadjoint, non-negative,
given e.g. by non-negative, real scalar $L_{\infty}$-multiplication-operators,
then this abstract framework yields – with well-chosen boundary conditions
– well-posedness of the problem, if we assume that $\mu$ and $\epsilon+\sigma$
are both strictly positive. This allows for a type change by having
$\epsilon=0$ in some regions (eddy current case) and $\epsilon$
strictly positive in others. This eddy current problem is well-understood
and well-justified, see \cite{MMA:MMA4515} or \cite[Section 5.3]{Wau16}.
The problem we want to investigate here goes, however, one step further.
We assume $\epsilon=0$ everywhere and $\sigma$ may still vanish
in some regions, as e.g. suggested in \cite{zbMATH06064782}.

In the case $\epsilon=0$, we eliminate $\mathrm{H}$ and obtain \textcolor{black}{
\begin{equation}
\partial_{0}\sigma\mathrm{E}+\curl\mu^{-1}\interior{\curl}\mathrm{E}=-\partial_{0}\mathrm{J}\label{eq:degen}
\end{equation}}
as a degenerate eddy current problem, which formally has parabolic
regions, where $\sigma$ is strictly positive, and elliptic regions,
where $\sigma$ vanishes. Note that this indeed represents a particularly
degenerate situation for if $\sigma$ vanishes in some regions, the
resulting problem still has a null-space, stemming from the infinite-dimensional
null-space of the $\curl$-operator. In the derivation to be carried
out below this is in fact the crucial observation.

In a sense the problems discussed in this manuscript can also be regarded
as the parabolic extension of the framework provided for elliptic
type problems presented in \cite{TW14}, where nonlinear differential
inclusions in divergence form have been discussed.

The extended abstract framework of \cite{PIC_2010:1889} still allows
us to incorporate the degenerate situation, where $\sigma$ is only
supported in a bounded subset $\Omega_{c}$ of the underlying open
set $\Omega$ (with positive distance to the boundary of $\Omega$). 

Although electromagnetic fields are generally accepted to be controlled
by Maxwell's equations, it is still well established with engineers,
see e.g. \cite{zbMATH04191042,Dirks1996}, to discard Maxwell's correction,
i.e., the displacement current term. It appears that the rigorous justification
of the above degenerate eddy current problem, where $\epsilon=0$
and $\sigma$ vanishes in some region, is still open or rather unattainable.

For a survey concerning the eddy current problem the reader may consult
\cite[Chapter 8]{zbMATH01266750} and for various variants \cite{RDS:RDS5861}.  \red{We shall furthermore refer to \cite{RFV2003,RPV2003,RV2008,RV2010} for the eddy current problem particularly considered in the time harmonic case. A convergence result relating the non-vanishing dielectricity case to the eddy current version of Maxwell's equations is also presented in \cite{RV2010}. We connect this convergence statement to the one derived in the concluding section of the paper at hand at the end of this manuscript, see Remark \ref{rem:Conv}. We refer to \cite{Rappaz} for a mathematical treatment of eddy current type problems and a selection of applications.}

\red{
For a treatment of the full time-dependent problem with nowhere vanishing $\sigma$, we refer to the recent paper \cite{francini2020existence}. This treatment prerequisites more assumptions on the smoothness of the boundary of the underlying domain (as well as on the magnetic permeability), which we wish to avoid here. 
}

More specifically, our investigation is inspired by a series of papers
by S. Nicaise et al., \cite{zbMATH06290947,zbMATH06418162,zbMATH06435646}. \red{Among other things the so-called $A$-$\varphi$ approach is addressed in these references. We shall comment on this approach, when we present the complete solution theory\footnote{\red{A (linear) solution theory (for a linear operator $B$) comprises not just a description of a class
of right-hand sides $f$ for which a solution $u$ of $Bu=f$ can be found, but also to
identify a complete linear space, in which the solution can be found. Furthermore, one needs to ensure that for every right-hand side $f$ produced by an element $u$ in the way that $Bu=f$, we actually can recover the original $u$ from this
right-hand side by applying the proposed solution procedure. Indeed,
here we consider providing a solution theory as establishing that
the operator is a continuous bijection between its domain and its
range as complete linear spaces (well-posedness).}} for the eddy current problem discussed here, see Remark \ref{rem:Aphi}.} 

We will employ the theory of evolutionary equations as laid out in
Section \ref{sec:A-Brief-Introduction}, see \cite{PDE_DeGruyter,PIC_2010:1889},
 to analyse the structure of the degenerate eddy current problem.
It will prove to be beneficial to embed the degenerate eddy current
problem into an abstract class of degenerate parabolic systems in
order to understand the mechanism of well-posedness more deeply. After
a brief introduction, Section \ref{sec:A-Brief-Introduction}, into
the theory of a problem class, which we will refer to as evolutionary
equations or evo-systems, we shall investigate the mentioned abstract
class of degenerate parabolic problems as a special case more closely
in Section \ref{sec:A-Class-of}.

The application to the degenerate eddy current problem is then given
in the concluding Section \ref{sec:Application-to-a}. In particular,
having reformulated and solved the degenerate eddy current type problem,
we shall address the validity of the equations one started out with.
It appears that this a posteriori justification of the original equation
has not been addressed in the literature as of yet. The application
to the eddy current type model is discussed further in the concluding
2 sections. There we present an alternative saddle-point formulation
for the problem at hand, which might be useful for numerical considerations.
In fact a similar strategy has led to an efficient numerical treatment
of Maxwell's equations (see \cite{Tas07}). Moreover, we shall justify
the degenerate eddy current model as a regular limit case of non-degenerate
problems. In the framework presented here, we are thus mathematically
justifying that the degenerate eddy current problem is indeed approachable
by regular problems so that the maybe-easier-to-solve degenerate parabolic
problem leads to an appropriate approximation of the full hyperbolic
Maxwell's equations.

\section{A Brief Introduction to Evo-Systems}

In this section we shall introduce the general abstract problem class
we like to use as the underlying structure of the derivations to come.

More precisely, we will discuss \emph{evolutionary equations}, \emph{evo-systems}
for short, in the following. These terms are chosen deliberately in
order to distinguish from classical (explicit) evolution equations,
which turn out to be just a special case of the class of evo-systems.
For convenience of the reader, we gather some necessary information
as follows.

The starting idea of the evo-system approach is to realise that the
time-differentiation can be established as a normal operator in a
real, weighted $L^{2}$-type Hilbert space $H_{\rho,0}(\mathbb{R};H)$
, $\rho\in\oi0\infty$ , see e.g. \cite{PDE_DeGruyter}, characterised
by
\[
H_{\rho,0}\left(\mathbb{R},H\right)=\left\{ f\in L^{2,\mathrm{loc}}\left(\mathbb{R},H\right)|\left|f\right|_{\rho,0,0}\::=\sqrt{\int_{\mathbb{R}}\left|f\left(t\right)\right|_{\textcolor{black}{H}}^{2}\exp\left(-2\rho t\right)\:dt}<\infty\right\} ,
\]
where $\left|\:\cdot\:\right|_{\textcolor{black}{H}}$ denotes the norm in the underlying
\emph{real} Hilbert space $H$. Our choice of a real Hilbert space
is no important constraint, it merely is an adjustment to account
for mostly real physical quantities. Note that every complex Hilbert
space is in fact a real Hilbert space if we restrict scalar multipliers
to $\mathbb{R}$ and take the real part of the inner product as the
real inner product.

The inner product $\left\langle \:\cdot\;|\:\cdot\;\right\rangle _{\rho,0,0}$
of $H_{\rho,0}\left(\mathbb{R},H\right)$ is given by 
\[
\left(\phi,\psi\right)\mapsto\int_{\mathbb{R}}\left\langle \phi\left(t\right)|\psi\left(t\right)\right\rangle _{\textcolor{black}{H}}\:\exp\left(-2\rho t\right)\:dt,
\]
where $\left\langle \:\cdot\;|\:\cdot\;\right\rangle _{\textcolor{black}{H}}$ denotes
the inner product of $H$. We define the time-derivative $\partial_{0,\rho}$
(or just $\partial_{0}$, if $\rho$ is clear from the context) to
be the distributional derivative with respect to the first variable
in $H_{\rho,0}(\mathbb{R},H)$ with maximal domain. We also put $H_{\rho,1}(\mathbb{R},H)\coloneqq D(\partial_{0})$
endowed with $\left\langle \partial_{0}\cdot\;|\partial_{0}\cdot\;\right\rangle _{\rho,0,0}$
as scalar product. This is a scalar product the induced norm of which
being equivalent to the graph norm of $D(\partial_{0})$. Indeed,
for this $\partial_{0}$ needs to be continuously invertible. This
property on the other hand follows from maximal accretivity of $\partial_{0}$.
In fact, a simple integration-by-parts procedure shows that
\[
\frac{1}{2}\overline{\left(\partial_{0}+\partial_{0}^{*}\right)}\eqqcolon\mathrm{sym}\text{\ensuremath{\left(\partial_{0}\right)}}\supseteq\frac{1}{2}\left(\partial_{0}+\partial_{0}^{*}\right)=\rho,
\]
where $\rho$ is a short-hand for the operator of multiplying by the
scalar value $\rho$. So $\partial_{0}$ is (real) strictly positive
definite (or accretive). This observation can be lifted to obtain
a solution theory for systems (evo-systems) of the form
\[
\left(\partial_{0}M\left(\partial_{0}^{-1}\right)+A\right)U=F,
\]
where here we focus on simple – so-called – `material law' operators
of the form
\[
M\left(\partial_{0}^{-1}\right)=M_{0}+\partial_{0}^{-1}M_{1},
\]
where $M_{k}$, $k\in\left\{ 0,1\right\} $, are certain continuous,
linear operators in $H$. The operator $A$ is densely defined and
closed in the Hilbert space $H$. All the operators $M_{0}$, $M_{1}$,
and $A$ are (canonically) lifted to the $H$-valued space $H_{\rho,0}(\mathbb{R};H)$
by being applied pointwise with maximal domain. Re-using the notation
for these lifted operators, we easily verify that $M_{0}$ and $M_{1}$
are still bounded linear operator in the extended space $H_{\rho,0}(\mathbb{R};H)$
even commuting with $\partial_{0},$ that is,
\[
M_{k}\partial_{0}\subseteq\partial_{0}M_{k}\quad(k\in\{0,1\}).
\]
$A$ acting in $H_{\rho,0}(\mathbb{R};H)$ will still be densely defined
and closed; the adjoint of the lifted $A$ is the lift of the adjoint
of $A$ having acted in $H$. Focusing on the simple material law
mentioned above, we want to solve evo-systems of the form
\begin{equation}
\left(\overline{\partial_{0}M_{0}+M_{1}+A}\right)U=F.\label{eq:evo-sys}
\end{equation}
By solving this evo-system, we mean to show that for all $F\in H_{\rho,0}(\mathbb{R};H)$
there exists a unique $U\in H_{\rho,0}(\mathbb{R};H)$ satisfying
\eqref{eq:evo-sys}. In other words, $\left(\overline{\partial_{0}M_{0}+M_{1}+A}\right)$
needs to be shown to be continuously invertible. 

Furthermore, in order to render \eqref{eq:evo-sys} `physically meaningful',
we shall show that \eqref{eq:evo-sys} also leads to a \emph{causal}
solution operator, which will be quantified in the next theorem and
roughly means that there is `no reaction' $U$, if there is `no action'
$F$. We shall furthermore refer to \cite{Wau15-1} and to \cite[Chapter 2]{Wau16}
for a more detailed account on causality.

The issue in the context of well-posedness of \eqref{eq:evo-sys},
that is, continuous invertibility of $\left(\overline{\partial_{0}M_{0}+M_{1}+A}\right)$
is, see e.g. \cite{PIC_2010:1889,Wau16}, to establish estimates of
the form
\begin{eqnarray}
\left\langle U|\left(\partial_{0}M_{0}+M_{1}+A\right)U\right\rangle _{\rho,0,0} & \geq & c_{0}\left\langle U|U\right\rangle _{\rho,0,0}\quad(U\in D(A)\cap D(\partial_{0})),\label{eq:pos-def00}
\end{eqnarray}
\begin{eqnarray}
\left\langle V|\left(\partial_{0}M_{0}+M_{1}+A\right)^{*}V\right\rangle _{\rho,0,0} & \geq & c_{0}\left\langle V|V\right\rangle _{\rho,0,0}\quad(V\in D(\left(\partial_{0}M_{0}+M_{1}+A\right)^{*}))\label{eq:posdef01}
\end{eqnarray}
for some $c_{0}>0$.

In the following we shall employ the convention to denote by $D(C),$
$R(C)$, $N(C)$ the domain, range and kernel of a linear operator
$C$.

We record the following variant of \cite[Theorem 2.3]{PIC_2010:1889}
or \cite[Theorem 3.4.6]{Wau16}. For this we briefly emphasise that
in contrast to earlier treatments of this theorem, we shall focus
on the real Hilbert space case, only. In this way the real-parts used
for the positive definiteness estimates in the mentioned theorems
can entirely be dispensed with.

\begin{thm}\label{-Solution-0} Let $M_{0},M_{1}\in L(H)$ with $M_{0}=M_{0}^{\ast}.$
Moreover, let $A:D\left(A\right)\subseteq H\to H$ be a closed, densely
defined linear operator such that 
\begin{equation}
\left\langle W|\left(\rho M_{0}+M_{1}+A\right)W\right\rangle _{\textcolor{black}{H}}\geq c_{0}\left\langle W|W\right\rangle _{\textcolor{black}{H}}\label{eq:posdefaa}
\end{equation}
\begin{equation}
\left\langle V|\left(\rho M_{0}+M_{1}^{\ast}+A^{*}\right)V\right\rangle _{\textcolor{black}{H}}\geq c_{0}\left\langle V|V\right\rangle _{\textcolor{black}{H}}\label{eq:posdefbb}
\end{equation}
for some $c_{0},\rho_{0}\in\oi0\infty$ and all $W\in D\left(A\right)$,
$V\in D\left(A^{*}\right)$ and $\rho\in\lci{\rho_{0}}\infty$. Then,
equation \eqref{eq:evo-sys} has for every $F\in H_{\rho,0}\left(\mathbb{R},H\right)$
a unique solution $U\in H_{\rho,0}\left(\mathbb{R},H\right)$. Moreover,
we have for the corresponding solution operator the estimate
\[
\left|\chi_{_{]-\infty,a]}}\left(\overline{\partial_{0}M_{0}+M_{1}+A}\right)^{-1}F\right|_{\rho,0,0}\leq\frac{1}{c_{0}}\;\left|\chi_{_{]-\infty,a]}}F\right|_{\rho,0,0}
\]
for all $a\in\mathbb{R}$ and $F\in H_{\rho,0}\left(\mathbb{R},H\right)$,
that is, we have continuous and causal dependence on the data.\end{thm}

\begin{proof}The result largely follows with the general results
in \cite{PIC_2010:1889} and is a special case of \cite[Theorem 3.4.6]{Wau16}
or of \cite[Theorem 3.1, Theorem 4.4]{Trost18}. Since, however, the
material law is more elementary here, we outline – for sake of transparency
and to remain self-contained – a more straightforward independent
proof. By density of $D(A)$ in $H$, we obtain that $D(A)$-valued
continuously differentiable functions with compact support are dense
in $H_{\rho,0}(\mathbb{R};H)$.

Thus, letting $U\in\mathring{C_{1}}(\mathbb{R};D(A))$ and using the Cauchy–Schwarz
inequality as well as integration by parts, we obtain
\begin{align}
 & \left|\chi_{_{]-\infty,a]}}U\right|_{\rho,0,0}\;\left|\chi_{_{]-\infty,a]}}\left(\partial_{0}M_{0}+M_{1}+A\right)U\right|_{\rho,0,0}\nonumber \\
 & \geq\left\langle \chi_{_{]-\infty,a]}}U|\left(\partial_{0}M_{0}+M_{1}+A\right)U\right\rangle _{\rho,0,0}\nonumber \\
 & =\int_{-\infty}^{a}\left\langle U|\left(\partial_{0}M_{0}+M_{1}+A\right)U\right\rangle _{\textcolor{black}{H}}\left(t\right)\;\exp\left(-2\rho t\right)\:dt\nonumber \\
 & =\int_{-\infty}^{a}\frac{1}{2}\left\langle U|M_{0}U\right\rangle _{\textcolor{black}{H}}^{\prime}\left(t\right)\:\exp\left(-2\rho t\right)\:dt+\int_{-\infty}^{a}\left\langle U|M_{1}U\right\rangle _{\textcolor{black}{H}}\left(t\right)\;\exp\left(-2\rho t\right)\:dt\nonumber \\
 & \quad\quad+\int_{-\infty}^{a}\left\langle U|AU\right\rangle _{\textcolor{black}{H}}\left(t\right)\;\exp\left(-2\rho t\right)\:dt\nonumber \\
 & =\frac{1}{2}\left\langle U|M_{0}U\right\rangle _{\textcolor{black}{H}}\left(a\right)\:\exp\left(-2\rho a\right)+\rho\int_{-\infty}^{a}\left\langle U|M_{0}U\right\rangle _{\textcolor{black}{H}}\left(t\right)\:\exp\left(-2\rho t\right)\:dt\label{eq:posdefxx}\\
 & \quad\quad+\int_{-\infty}^{a}\left\langle U|M_{1}U\right\rangle _{\textcolor{black}{H}}\left(t\right)\;\exp\left(-2\rho t\right)\:dt+\int_{-\infty}^{a}\left\langle U|AU\right\rangle _{\textcolor{black}{H}}\left(t\right)\;\exp\left(-2\rho t\right)\:dt\nonumber \\
 & \geq\rho\int_{-\infty}^{a}\left\langle U|M_{0}U\right\rangle _{\textcolor{black}{H}}\left(t\right)\:\exp\left(-2\rho t\right)\:dt+\int_{-\infty}^{a}\left\langle U|M_{1}U\right\rangle _{\textcolor{black}{H}}\left(t\right)\;\exp\left(-2\rho t\right)\:dt\nonumber \\
 & \quad\quad+\int_{-\infty}^{a}\left\langle U|AU\right\rangle _{\textcolor{black}{H}}\left(t\right)\;\exp\left(-2\rho t\right)\:dt\nonumber \\
 & =\int_{-\infty}^{a}\left\langle U|\left(\rho M_{0}+M_{1}+A\right)U\right\rangle _{\textcolor{black}{H}}\left(t\right)\:\exp\left(-2\rho t\right)\:dt\nonumber \\
 & \geq c_{0}\left\langle \chi_{_{]-\infty,a]}}U|\chi_{_{]-\infty,a]}}U\right\rangle _{\rho,0,0}.\nonumber 
\end{align}
Letting $a\to\infty$ in \eqref{eq:posdefxx} we get \eqref{eq:pos-def00}
with a density argument. Similarly, we obtain \eqref{eq:posdef01}
by re-doing the above estimate for $a=\infty$ and $A$ replaced by
$A^{*}$ (in which case there is no point-evaluation at the upper
time boundary value and we need to confirm that $\left(\partial_{0}M_{0}+M_{1}+A\right)^{*}=\overline{\partial_{0}^{*}M_{0}+M_{1}^{*}+A^{*}}$,
which in turn follows using suitable density arguments as for instance
in \cite[the proof of Theorem 2.13]{PTW_Evo}). Thus $\left(\overline{\partial_{0}M_{0}+M_{1}+A}\right)^{-1}$
is continuous. Hence, from $N\left(\left(\partial_{0}M_{0}+M_{1}+A\right)^{*}\right)=N\left(\overline{\partial_{0}^{*}M_{0}+M_{1}^{*}+A^{*}}\right)=\left\{ 0\right\} $,
we infer that $\left(\overline{\partial_{0}M_{0}+M_{1}+A}\right)^{-1}$
is also everywhere defined. Moreover, the above estimate \eqref{eq:posdefxx}
shows
\begin{equation}
\left|\chi_{_{]-\infty,a]}}\left(\overline{\partial_{0}M_{0}+M_{1}+A}\right)^{-1}F\right|_{\rho,0,0}\leq\frac{1}{c_{0}}\;\left|\chi_{_{]-\infty,a]}}F\right|_{\rho,0,0}\label{eq:causal-estimate}
\end{equation}
for all $a\in\mathbb{R}$ and $F\in H_{\rho,0}\left(\mathbb{R},H\right)$.
If $F=0$ on the time interval $]-\infty,a]$ then we read off that
also the solution $U$ must vanish on this time-interval, i.e. we
have causality. Letting $a\to\infty$ in \eqref{eq:causal-estimate}
shows continuous dependence in the form 
\[
\left\Vert \overline{\textcolor{black}{(}\partial_{0}M_{0}+M_{1}+A\textcolor{black}{)}}^{-1}\right\Vert \leq\frac{1}{c_{0}}.\tag*{{\qedhere}}
\]
\end{proof}

\begin{rem}\label{Dual-space}We identify the dual spaces
\begin{eqnarray*}
H & = & H^{\prime},\\
H_{\rho,0}\left(\mathbb{R}\right) & = & H_{\rho,0}\left(\mathbb{R}\right)^{\prime},
\end{eqnarray*}
and so we have
\begin{eqnarray*}
H_{\rho,0}\left(\mathbb{R},H\right) & = & H_{\rho,0}\left(\mathbb{R},H\right)^{\prime}.
\end{eqnarray*}
Moreover, the dual $\left(\partial_{0}^{*}\right)^{\diamond}$ of
the --- by choice of inner product --- unitary operator $\partial_{0}^{*}\iota_{H_{\rho,1}\left(\mathbb{R},H\right)}:H_{\rho,1}\left(\mathbb{R},H\right)\to H_{\rho,0}\left(\mathbb{R},H\right)$ --- has an extension to a continuous operator for which we keep the
notation $\partial_{0}$ and so 
\begin{eqnarray*}
\partial_{0}:H_{\rho,0}\left(\mathbb{R},H\right) & \to & H_{\rho,-1}\left(\mathbb{R},H\right):=H_{\rho,1}\left(\mathbb{R},H\right)^{\prime}.
\end{eqnarray*}
Similarly, the continuous mapping
\begin{eqnarray*}
A^{*}\iota_{H_{\rho,0}\left(\mathbb{R},D\left(A^{*}\right)\right)}:H_{\rho,0}\left(\mathbb{R},D\left(A^{*}\right)\right) & \to & H_{\rho,0}\left(\mathbb{R},H\right)
\end{eqnarray*}
has as dual
\begin{eqnarray*}
\left(A^{*}\right)^{\diamond}=\left(A^{*}\iota_{H_{\rho,0}\left(\mathbb{R},D\left(A^{*}\right)\right)}\right)^{\diamond}:H_{\rho,0}\left(\mathbb{R},H\right) & \to & H_{\rho,0}\left(\mathbb{R},D\left(A^{*}\right)^{\prime}\right),
\end{eqnarray*}
which may be considered as a continuous extension of $A$ and so justifies
(with some care) to keep $A$ as a notation for $\left(A^{*}\right)^{\diamond}$.
\footnote{Note that we routinely use $D\left(A\right)$ for the domain of $A$
also for the corresponding Hilbert space with respect to the graph
inner product of $A$. In this sense $D\left(A^{*}\right)^{\prime}$
denotes the dual Hilbert space of the Hilbert space $D\left(A^{*}\right)$.}

Indeed, for $\Psi\in H_{\rho,0}(\mathbb{R},D(A^{*}))$ we compute
\[
\left(\left(A^{*}\right)^{\diamond}\Phi\right)\left(\Psi\right)\coloneqq\left\langle \Psi|\left(A^{*}\right)^{\diamond}\Phi\right\rangle _{\rho,0,0}=\left\langle A^{*}\Psi|\Phi\right\rangle _{\rho,0,0}
\]
for all $\Phi\in H_{\rho,0}\left(\mathbb{R},D\left(A\right)\right)$,
in which case $A\Phi=\left(A^{*}\right)^{\diamond}\Phi$ and by continuous
extension also to $\Phi\in H_{\rho,0}\left(\mathbb{R},H\right)$.
We have for a solution of the evo-system \eqref{eq:evo-sys} that
\[
\partial_{0}M_{0}U+M_{1}U+AU=F
\]
holds in the space $H_{\rho,-1}\left(\mathbb{R},D\left(A^{*}\right)^{\prime}\right).$
Note that 
\[
H_{\rho,-1}\left(\mathbb{R},D\left(A^{*}\right)^{\prime}\right)\supseteq H_{\rho,-1}\left(\mathbb{R},H\right)\cap H_{\rho,0}\left(\mathbb{R},D\left(A^{*}\right)^{\prime}\right).
\]
We shall use this observation to conveniently drop the closure bar
in equations of the form \eqref{eq:evo-sys}.\end{rem}

\begin{rem}\label{local-in-time}\textcolor{black}{(a)} In the case of a simple material
law as used here it is interesting to note that the result easily
carries over to a local-in-time formulation. Indeed, the time-derivative
restricted to a finite time-interval $\left[0,T\right]$, $T\in\oi0\infty$,
given as the closure $\partial_{0,\rho,]0,T]}$ of $\partial_{0}$
restricted to $\interior C_{1}\left(]0,T],H\right)$ in $H_{\rho,0}\left(\oi0T,H\right)$
loses the skew-selfadjointness, keeps, however, the maximal accretivity.
We emphasise the parentheses of the interval in the index of the time-derivative
operator: $\partial_{0,\rho,]0,T]}$ has a zero boundary condition
at $0$, and no boundary condition at $T$; whereas $\partial_{0,\rho,[0,T[}$
(defined as $\partial_{0,\rho,]0,T]}$ with $]0,T]$ being interchanged
by $[0,T[$) has no boundary condition at $0$ and a zero boundary
condition at $T$. In classical terms, we have
\begin{eqnarray*}
D(\partial_{0,\rho,]0,T]}) & = & \{\phi\in H^{1}(0,T);\phi(0)=0\},\\
D(\partial_{0,\rho,[0,T[}) & = & \{\phi\in H^{1}(0,T);\phi(T)=0\}.
\end{eqnarray*}
For the closure $\partial_{0,\rho,[0,T[}$ we still have $\partial_{0,\rho,]0,T]}^{*}=-\partial_{0,\rho,[0,T[}+2\rho$.
Thus, it is rather straightforward to see 
\[
\partial_{0,\rho,]0,T]},\partial_{0,\rho,]0,T]}^{*}\geq\rho,
\]
which allows the solution theory of
\[
\partial_{0,\rho}M_{0}+M_{1}+A
\]
to be carried over to
\[
\partial_{0,\rho,]0,T]}M_{0}+M_{1}+A.
\]
In this sense the above solution strategy also carries over to problems
with finite time horizon. For this, we also refer to \cite{key-10}
for a numerical treatment of evo-systems. Regarding numerics, we shall
furthermore refer to the Section \ref{sec:An-Extended-System}.

\textcolor{black}{(b) It is also possible to use the above derived solution theory for incorporating initial value problems. For this there are at least two possibilities. One is to require that the initial datum $U_0$ is in the domain of $A$. Then one can show that for the unique solution $V$ of 
\[
   \overline{ (\partial_0 M_0 +M_1 +A)}V=-\chi_{[0,\infty)}M_1U_0-\chi_{[0,\infty)}AU_0,
\]it follows that $U=V+\chi_{[0,\infty)}U_0$ satisfies the initial value problem
\[
  \begin{cases}
     (\partial_0 M_0 +M_1 +A)U=0 & \text{ on }(0,\infty) \\
     (M_0U)(0+)=M_0U_0&
  \end{cases}
  \]
  in an appropriate sense. It is also possible to extend the solution operator $(\overline{\partial_0 M_0 +M_1 +A})^{-1}$ to a continuous linear operator $S$ from $H_{\rho,-1}(\mathbb{R};H)$ into itself. It can then be shown that the solution $U$ of the just introduced initial value problem satisfies
  \[
     U= S\delta_0M_0U_0,
  \]where $\delta_0$ is the Dirac delta-distribution. Interestingly, the latter formulation is also well-defined for $U_0\notin D(A)$ and, thus, serves as a generalisation for the initial value problem for less regular initial data; we refer to \cite[Chapter 6]{PDE_DeGruyter}, \cite[Lecture 9]{STW20} for the details.}
\end{rem}

Our focus in the following will be on a rather particular subclass,
where $M_{1}=0$ and $A=C^{*}C$ for a closed, densely defined operator
$C$ with closed range. The coefficient $M_{0}$ may have a non-trivial
null space but, as we shall see, that $0$ is in the resolvent set
of the reduction of $C^{*}C$ to the subspace $R\left(C^{*}\right)$,
which is also closed, can be used to compensate for this short-coming.
Recall that for elliptic problems, that is, for $M_{0}=0$, the strategy
of projecting onto $R(C^{*})$ has been successfully applied also
to non-linear (abstract) differential inclusions, see \cite{TW14}.
Also in \cite{TW14}, the crucial assumption for the well-posedness
was a closed range condition.

\section{\label{sec:A-Class-of}A Class of Degenerate Abstract Parabolic Equations }

In this whole section, we let $H$ and $X$ be Hilbert spaces and
let $\eta\in L(H)$ be a bounded, selfadjoint, non-negative operator.
Furthermore, let 
\[
C:D\left(C\right)\subseteq H\to X
\]
be closed and densely defined; throughout assume $C$ to have a closed
range.

Abstractly speaking, we want to consider 

\begin{equation}
\left(\overline{\partial_{0}\eta+C^{*}C}\right)U=F.\label{eq:basic-eq}
\end{equation}

\begin{rem}\label{NoteC}Note that the equation holds in the form
\[
\partial_{0}\eta U+C^{*}C\:U=F
\]
if considered in the space 
\[
H_{\rho,-1}\left(\mathbb{R},D\left(C^{*}C\right)^{\prime}\right).
\]
This is clear from Remark \ref{Dual-space}. Henceforth, we shall
therefore dispose of the closure bar in equations of the form \eqref{eq:basic-eq}
unless it is needed for sake of clarity. \end{rem}

Without having looked at this equation in detail, it is immediately
clear, where degeneracies might arise. Indeed, if $U$ attains non-zero
values in $N(\eta)\cap N(C)$, that is, if $U\in H_{\rho,0}\left(\mathbb{R},N\left(\eta\right)\cap N\left(C\right)\right)$
we have
\[
\partial_{0}\eta U+C^{*}CU=0,
\]
and so if $N\left(\eta\right)\cap N\left(C\right)$ is not trivial,
well-posedness for \eqref{eq:basic-eq} is out of reach. Hence, the
term `degenerate'. We shall come back to this issue in a moment's
time. Following the solution strategy for evo-systems as it has been
sketched in the previous section, we realise that the issue in the
context of well-posedness is to establish estimates of the form
\begin{eqnarray*}
\left\langle U|\left(\partial_{0}\eta+C^{*}C\right)U\right\rangle _{\rho,0,0} & = & \left\langle \eta^{1/2}U|\partial_{0}\eta^{1/2}U\right\rangle _{\rho,0,0}+\left\langle CU|CU\right\rangle _{\rho,0,0}\\
 &  & \geq c_{0}\left\langle U|U\right\rangle _{\rho,0,0},
\end{eqnarray*}
\begin{eqnarray*}
\left\langle U|\left(\partial_{0}\eta+C^{*}C\right)^{*}U\right\rangle _{\rho,0,0} & = & \left\langle \eta^{1/2}U|\partial_{0}^{*}\eta^{1/2}U\right\rangle _{\rho,0,0}+\left\langle CU|CU\right\rangle _{\rho,0,0}\\
 &  & \geq c_{0}\left\langle U|U\right\rangle _{\rho,0,0}.
\end{eqnarray*}

Since, due to the density of elements with compact time support in
$D\left(\partial_{0}\right)$,
\[
\left\langle \eta^{1/2}U|\partial_{0}^{*}\eta^{1/2}U\right\rangle _{\rho,0,0}=\left\langle \eta^{1/2}U|\partial_{0}\eta^{1/2}U\right\rangle _{\rho,0,0}=\rho\left|\eta^{1/2}U\right|_{\rho,0,0}^{2}
\]
we only need to consider one of the estimates, thus we need to have
\begin{equation}
\rho\left|\eta^{1/2}U\right|_{\rho,0,0}^{2}+\left|CU\right|_{\rho,0,0}^{2}\geq c_{0}\left|U\right|_{\rho,0,0}^{2},\label{eq:pddeg}
\end{equation}
which again emphasises that the Hilbert space we choose $U$ from
cannot contain the space $H_{\rho,0}\left(\mathbb{R},N\left(\eta\right)\cap N\left(C\right)\right)$.

It is the aim of this section to show that restricting our attention
to the orthogonal complement of $N\left(\eta\right)\cap N\left(C\right)$
as well as assuming an estimate of the type \eqref{eq:pddeg} for
$U$ attaining values in 
\[
H_{0}\coloneqq\left(N\left(\eta\right)\cap N\left(C\right)\right)^{\perp}\subseteq H
\]
leads to well-posedness and causality with state space $H_{0}$. Since
both $\eta$ and $C$ are operators acting on the `spatial' Hilbert
space, only, it is possible to provide an equivalent formulation,
which only uses the spatial scalar product.

\begin{prop}\label{prop:ptw} Let $C$ and $\eta$ be as above. Then
the following conditions are equivalent:\begin{enumerate}\item There
exists $\rho>0$ and $c_{0}>0$ such that for all $U\in H_{\rho,0}\left(\mathbb{R},H_{0}\cap D(C)\right)$
we have
\[
\rho\left|\eta^{1/2}U\right|_{\rho,0,0}^{2}+\left|CU\right|_{\rho,0,0}^{2}\geq c_{0}\left|U\right|_{\rho,0,0}^{2}.
\]

\item There exists $c_{0}>0$ such that for all $U\in H_{0}\cap D(C)$
we have
\[
\left|\eta^{1/2}U\right|_{H}^{2}+\left|CU\right|_{X}^{2}\geq c_{0}\left|U\right|_{H}^{2}.
\]

\end{enumerate}

\end{prop}

\begin{proof} An easy density argument implies that the second inequality
implies the first one with $\rho=1$ and the same $c_{0}>0$. Thus, it
remains to show the converse implication. For this, note that with
$\rho_{*}\coloneqq\max\{\rho,1\}$ we have for all $U\in H_{\rho,0}\left(\mathbb{R},H_{0}\cap D(C)\right)$
\[
\left|\eta^{1/2}U\right|_{\rho,0,0}^{2}+\left|CU\right|_{\rho,0,0}^{2}\geq\frac{c_{0}}{\rho_{*}}\left|U\right|_{\rho,0,0}^{2}.
\]
Let $x\in H_{0}\cap D(C)$. Using the latter inequality for $U(t)\coloneqq\exp(\rho t)x$
for $t\in[0,1]$ and $U(t)=0$ for $t<0$ and $t>1,$ we infer the
desired inequality.
\end{proof}

Next, note that, since elements in $N\left(\eta\right)\cap N\left(C\right)$
are orthogonal to $R\left(C^{*}\right)$ and $R\left(\eta\right)$
and if $C$ and consequently $C^{*}$ are operators with closed range
we may reduce the operator $C$ to $H_{0}\coloneqq\left(N\left(\eta\right)\cap N\left(C\right)\right)^{\perp}$.
Indeed, as we shall see next, the operators
\begin{eqnarray*}
C_{0}:D\left(C\right)\cap H_{0}\subseteq H_{0} & \to & X\\
u & \mapsto & Cu
\end{eqnarray*}
retains the closedness of the range and is also still densely defined.
With $\iota_{H_{0}\to H}$ denoting the canonical isometric embedding
of $H_{0}$ as a subspace of $H$, we have
\[
C_{0}=C\iota_{H_{0}\to H}.
\]
 The mentioned properties of $C_{0}$ are proved next.

\begin{lem}The operator $C_{0}$ is closed, densely defined and has
a closed range. \end{lem}

\begin{proof}It is 
\[
H=H_{0}\oplus H_{0}^{\perp}
\]
and 
\[
H_{0}^{\perp}=N\left(\eta\right)\cap N\left(C\right)\subseteq N\left(C\right)\subseteq D\left(C\right)
\]
and so
\[
D\left(C\right)=\left(D\left(C\right)\cap H_{0}\right)\oplus H_{0}^{\perp}.
\]
The density of $D\left(C_{0}\right)=D\left(C\right)\cap H_{0}$ in
$H_{0}$ now follows from the continuity of the orthogonal projector
$P_{H_{0}}$ onto $H_{0}$. Indeed, let $x_{\infty}\in H_{0}.$ Then
we find a sequence $\left(x_{n}\right)_{n}$ in $D(C)$ such that
$x_{n}\to x_{\infty}.$ Thus, also $P_{H_{0}}x_{n}\to P_{H_{0}}x_{\infty}=x_{\infty}.$
Since, $(1-P_{H_{0}})x_{n}\in D(C)$ for all $n\in\mathbb{N}$ by
the argument above, we infer that $(P_{H_{0}}x_{n})_{n\in\mathbb{N}}$
is, in fact, a sequence in $D(C_{0})$ showing that $C_{0}$ is densely
defined. 

Since $C_{0}=C\cap(H_{0}\oplus X),$ where we identify the operators
with their graphs, the closedness of $C_{0}$ follows.

We are left with showing the closedness of the range of $C_{0}$.
For this, let $z$ be a sequence in $H_{0}$ such that $C_{0}z=Cz\to w_{\infty}$
for some $w_{\infty}\in X$. Then by the closedness of the range of
$C$ we have
\[
Cx_{*}=w_{\infty}
\]
for some $x_{*}\in D\left(C\right)$. Since 
\[
w_{\infty}=Cx_{*}=CP_{H_{0}}x_{*}=C_{0}\left(P_{H_{0}}x_{*}\right),
\]
we confirm that $w_{\infty}\in R(C_{0})$ finally proving that indeed
closedness of the range is preserved.\end{proof} 

\begin{lem} We have
\[
C_{0}^{*}=\overline{\iota_{H_{0}\to H}^{*}C^{*}}.
\]

\end{lem}

\begin{proof} Since $C_{0}$ is densely defined we obtain the assertion
with \cite[Theorem 1.2]{Mother2013}.\end{proof}

Thus, we are led to study the reduced – by construction injective
– operator
\begin{align*}
\partial_{0}\eta_{0}+C_{0}^{*}C_{0} & =\iota_{H_{0}\to H}^{*}\left(\partial_{0}\eta+C^{*}C\right)\iota_{H_{0}\to H}
\end{align*}
with $\eta_{0}\coloneqq\iota_{H_{0}\to H}^{*}\eta\iota_{H_{0}\to H}$
now being selfadjoint in $H_{0}$. 

To proceed with our approach we need to assume moreover for some $c_{1}>0$
\begin{align}
 & \left|\eta_{0}^{1/2}U\right|_{H_{0}}^{2}+\left|C_{0}U\right|_{X}^{2}\geq c_{1}\left|U\right|_{H_{0}}^{2}\label{eq:assume}
\end{align}
for all $U\in D\left(C_{0}\right)$. 

\begin{rem} Note that \eqref{eq:assume} is equivalent to the inequalities
asserted in Proposition \ref{prop:ptw}. For this, we observe that
for all $U\in H_{0}\cap D(C)=D(C_{0})$ we have $C_{0}U=CU$. Moreover,
for $U\in H_{0}$ we compute
\begin{align*}
\left|\eta_{0}^{1/2}U\right|_{H_{0}}^{2} & =\langle\eta_{0}^{1/2}U|\eta_{0}^{1/2}U\rangle_{H_{0}}\\
 & =\langle U|\eta_{0}U\rangle_{H_{0}}\\
 & =\langle U|\iota_{H_{0}\to H}^{*}\eta\iota_{H_{0}\to H}U\rangle_{H_{0}}\\
 & =\langle\iota_{H_{0}\to H}U|\eta\iota_{H_{0}\to H}U\rangle_{H_{0}}\\
 & =\langle U|\eta U\rangle_{H}\\
 & =\left|\eta^{1/2}U\right|_{H}^{2},
\end{align*}
which yields the desired equivalence.\end{rem}

The latter assumption leads to well-posedness of the evo-system under
consideration in the state space $H_{0}$.

\begin{prop}Assume \eqref{eq:assume}. Then $\overline{\partial_{0}\eta_{0}+C_{0}^{*}C_{0}}$
is continuously invertible in $H_{\rho,0}(\mathbb{R},H_{0})$ for
all $\rho\geq1$.

\end{prop}

\begin{proof} In order to prove this theorem, it suffices to apply
Theorem \ref{-Solution-0} to $M_{0}=\eta_{0}$ and $A=C_{0}^{*}C_{0}$
note that it is easy to see that the positive definiteness conditions
of Theorem \ref{-Solution-0} are then satisfied due to assumption
\eqref{eq:assume}.\end{proof}

The next result relates the solution $U$ of

\begin{equation}
\left(\overline{\partial_{0}\eta_{0}+C_{0}^{*}C_{0}}\right)U=f\label{eq:evo-rig}
\end{equation}
or
\begin{equation}
\left(\eta+C_{0}^{*}C_{0}\partial_{0}^{-1}\right)U=\partial_{0}^{-1}f\label{eq:basic-rigorous}
\end{equation}
in $H_{0}$ to the equation \eqref{eq:basic-eq}.

\begin{prop}\label{prop:just} Assume \eqref{eq:assume}. Let $U\coloneqq\left(\overline{\partial_{0}\eta_{0}+C_{0}^{*}C_{0}}\right)^{-1}F$
for some $F\in H_{\rho,0}(\mathbb{R},H_{0})$ for some $\rho\geq1$.
Then $U$ satisfies \eqref{eq:basic-eq}.

\end{prop}

\begin{proof} Since $\partial_{0}^{-1}$ commutes with $\left(\overline{\partial_{0}\eta_{0}+C_{0}^{*}C_{0}}\right)^{-1}$,
we infer that \eqref{eq:basic-rigorous} is in fact a consequence
of \eqref{eq:evo-rig}. Moreover, we read off that
\[
\partial_{0}^{-1}U\in D\left(C_{0}^{*}C_{0}\right)
\]
and so in particular
\begin{equation}
C_{0}^{*}C_{0}\partial_{0}^{-1}U=C^{*}C\partial_{0}^{-1}U\label{eq:eqcc0}
\end{equation}
and 
\begin{equation}
\eta_{0}U=\eta U.\label{eq:eqeta}
\end{equation}

Indeed, since
\begin{eqnarray*}
\left\langle \phi|C_{0}^{*}C_{0}\partial_{0}^{-1}U\right\rangle _{H} & = & \left\langle C_{0}\phi|C_{0}\partial_{0}^{-1}U\right\rangle _{X}\\
 & = & \left\langle C\phi|C\partial_{0}^{-1}U\right\rangle _{X}
\end{eqnarray*}
 and
\begin{eqnarray*}
\left\langle \phi|\eta_{0}U\right\rangle _{H} & = & \left\langle \phi|\eta U\right\rangle _{H}
\end{eqnarray*}
for all $\phi\in D\left(C_{0}\right)=D\left(C\right)\cap H_{0}$,
as well as
\begin{eqnarray*}
\left\langle \psi|C_{0}^{*}C_{0}\partial_{0}^{-1}U\right\rangle _{H} & = & \left\langle C\psi|C\partial_{0}^{-1}U\right\rangle _{X}=0
\end{eqnarray*}
and 
\[
\left\langle \psi|\eta_{0}U\right\rangle _{H}=\left\langle \psi|\eta U\right\rangle _{H}=0
\]
for $\psi\in H_{0}^{\perp}=N\left(C\right)\cap N\left(\eta\right)$,
we have
\begin{eqnarray*}
\left\langle V|C_{0}^{*}C_{0}\partial_{0}^{-1}U\right\rangle _{H} & = & \left\langle CV|C\partial_{0}^{-1}U\right\rangle _{X}
\end{eqnarray*}
 and
\begin{eqnarray*}
\left\langle V|\eta_{0}U\right\rangle _{H} & = & \left\langle V|\eta U\right\rangle _{H}
\end{eqnarray*}
for all $V\in D\left(C\right)$. Thus, we read off \eqref{eq:eqeta}
and
\[
C\partial_{0}^{-1}U\in D\left(C^{*}\right)
\]
as well as
\[
C_{0}^{*}C_{0}\partial_{0}^{-1}U=C^{*}C\partial_{0}^{-1}U,
\]
that is, \eqref{eq:eqcc0}. 

Letting now
\[
V\coloneqq-C\partial_{0}^{-1}U
\]
we obtain
\begin{equation}
\begin{array}{rcl}
V+C\partial_{0}^{-1}U & = & 0,\\
\eta U-C^{*}V & = & \partial_{0}^{-1}F.
\end{array}\label{eq:Max-sim}
\end{equation}
Thus, we find that
\begin{equation}
\begin{array}{rcl}
\partial_{0}V+CU & = & 0,\\
\eta U-C^{*}V & = & \partial_{0}^{-1}F,
\end{array}\label{eq:Max-simsim}
\end{equation}
and so also
\[
\partial_{0}\eta U+C^{*}CU=F
\]
hold in a distributional sense. In particular, this confirms that
we have indeed solved the original equation \eqref{eq:basic-eq}.\end{proof}

\begin{rem}\label{rem:reg} For $F\in H_{\rho,0}(\mathbb{R};H_{0})$
we set $U\coloneqq(\partial_{0}\eta_{0}+C_{0}^{\ast}C_{0})^{-1}F\in H_{\rho,0}(\mathbb{R};H).$
Then, there exists a sequence $(U_{n})_{n\in\mathbb{N}}$ in $H_{\rho,1}(\mathbb{R};D(C_{0}^{\ast}C_{0}))$
such $U_{n}\to U$ and $(\partial_{0}\eta_{0}+C_{0}^{\ast}C_{0})U_{n}\to F.$
For $n\in\mathbb{N}$ we estimate 
\begin{align}
|C_{0}U_{n}|_{\rho,0,0}^{2} & \leq\rho\langle\eta_{0}U_{n}|U_{n}\rangle_{\rho,0,0}+\langle C_{0}^{\ast}C_{0}U_{n}|U_{n}\rangle_{\rho,0,0}\nonumber \\
 & =\langle\left(\partial_{0}\eta_{0}+C_{0}^{\ast}C_{0}\right)U_{n}|U_{n}\rangle_{\rho,0,0}\label{eq:Un}
\end{align}
and since the right-hand side is bounded, we infer that (up to a subsequence)
$C_{0}U_{n}\rightharpoonup w$ for some $w\in H_{\rho,0}(\mathbb{R};X).$
By the closedness (and hence, weak closedness) of $C_{0},$ we derive
that $U\in D(C_{0})$ and $w=C_{0}u.$ In particular $|C_{0}U|_{\rho,0,0}\leq\liminf_{n\to\infty}|C_{0}U_{n}|_{\rho,0,0}.$
Letting $n$ tend to infinity in (\ref{eq:Un}) we get 
\begin{align*}
\langle F|U\rangle_{\rho,0,0} & =\rho\langle\eta_{0}U|U\rangle_{\rho,0,0}+\lim_{n\to\infty}\langle C_{0}U_{n}|C_{0}U_{n}\rangle_{\rho,0,0}\\
 & \geq\rho\langle\eta_{0}U|U\rangle_{\rho,0,0}+\langle C_{0}U|C_{0}U\rangle_{\rho,0,0}\\
 & =\frac{1}{2}\left(2\rho\langle\eta_{0}U|U\rangle_{\rho,0,0}+\langle C_{0}U|C_{0}U\rangle_{\rho,0,0}\right)+\frac{1}{2}\langle C_{0}U|C_{0}U\rangle_{\rho,0,0}\\
 & \geq\frac{1}{2}c_{1}|U|_{\rho,0,0}^{2}+\frac{1}{2}|C_{0}U|_{\rho,0,0}^{2}\\
 & \geq\tilde{c}_{1}|U|_{\rho,0,1}^{2}
\end{align*}
with $\tilde{c}_{1}\coloneqq\frac{1}{2}\min\{1,c_{1}\}.$ Estimating
the left hand side by $|F|_{\rho,0,-1}|U|_{\rho,0,1}$ we end up with
\[
|U|_{\rho,0,1}\leq\frac{1}{\tilde{c}_{1}}|F|_{\rho,0,-1}.
\]
Thus, the solution operator $S$ attains values in $H_{\rho,0}(\mathbb{R};D(C_{0}))$
and can be extended continuously to $H_{\rho,0}(\mathbb{R};D(C_{0}^{\ast})').$
This is a refinement of the earlier observation in the general case,
see Remarks \ref{Dual-space} and \ref{NoteC}.

\end{rem}

\begin{example} As a quick example, it might be illustrative to apply
the observations in the previous remark to the (non-degenerate) case
of the heat equation. So, take $\eta=1$ to be the identity in $H=L^{2}(\Omega)$
and $C=\interior\grad$ with $D(C)=H_{0}^{1}(\Omega)$. Then the previous
remark confirmed a solution theory for the heat equation $\left(\partial_{0}-\Delta\right)U=F$
for right-hand sides $F$ taking values in $H^{-1}(\Omega)$. By the
general theory developed here, we obtain that $U$ assumes values
even in $H_{0}^{1}(\Omega)$. 

\end{example}

For sake of later reference let us summarise the core of the above
observations in the following theorem.

\begin{thm}\label{thm:Solution} Let $C:D\left(C\right)\subseteq H\to X$
be a closed densely defined linear operator with closed range and
such that \eqref{eq:assume} holds. Then, for every $F\in H_{\rho,0}\left(\mathbb{R},D(C_{0}^{\ast})'\right)$
there is a unique (weak) solution $U\in H_{\rho,0}\left(\mathbb{R},D\left(C_{0}\right)\right)$
of \eqref{eq:basic-rigorous} or equivalently of the system \eqref{eq:Max-sim}.
Moreover the solution operator $S:H_{\rho,0}\left(\mathbb{R},H_{0}\right)\to H_{\rho,0}\left(\mathbb{R},D\left(C_{0}\right)\right)$
is continuous ($\left|\:\cdot\:\right|_{\rho,0,1}$ denotes the norm
of $H_{\rho,0}\left(\mathbb{R},D\left(C_{0}\right)\right)$ and causal
in the sense that
\[
\left|\chi_{_{]-\infty,a]}}SF\right|_{\rho,0,1}\leq C_{1}\:\left|\chi_{_{]-\infty,a]}}F\right|_{\rho,0,-1}
\]
for some positive $C_{1}$ uniformly in $a\in\mathbb{R}$ and $F\in H_{\rho,0}\left(\mathbb{R},H_{0}\right)$
as long as $\rho\in\oi0\infty$ is sufficiently large.\end{thm}

\begin{proof}The result largely follows from our previous considerations.
The sharper regularity statement $U\in H_{\rho,0}\left(\mathbb{R},D\left(C_{0}\right)\right)$
and the sharper continuous dependence statement follows by Remark
\ref{rem:reg}. The claim of causality follows from a slight refinement
of the estimates along the reasoning of Remark \ref{rem:reg}. Indeed,
we have for all sufficiently large $\rho\in\oi0\infty$
\begin{align*}
 & \left|\chi_{_{]-\infty,a]}}U\right|_{\rho,0,1}\:\left|\chi_{_{]-\infty,a]}}\left(\partial_{0}\eta_{0}+C_{0}^{*}C_{0}\right)U\right|_{\rho,0,-1}\\
 & \geq\left\langle \chi_{_{]-\infty,a]}}U|\left(\partial_{0}\eta_{0}+C_{0}^{*}C_{0}\right)U\right\rangle _{\rho,0,0}\\
 & =\rho\left|\eta_{0}^{1/2}\chi_{_{]-\infty,a]}}U\right|_{\rho,0,0}^{2}+\frac{1}{2}\left|\eta_{0}^{1/2}U\left(a\right)\right|_{0}^{2}\exp\left(-2\rho a\right)+\left|C_{0}\chi_{_{]-\infty,a]}}U\right|_{\rho,0,0}^{2}\\
 & \geq\frac{1}{2}c_{1}\left|\chi_{_{]-\infty,a]}}U\right|_{\rho,0,0}^{2}+\frac{1}{2}\left|C_{0}\chi_{_{]-\infty,a]}}U\right|_{\rho,0,0}^{2}\\
 & \geq\frac{1}{2}c_{1}\left(\left|\chi_{_{]-\infty,a]}}U\right|_{\rho,0,0}^{2}+\left|C_{0}\chi_{_{]-\infty,a]}}U\right|_{\rho,0,0}^{2}\right)=\frac{1}{2}c_{1}\left|\chi_{_{]-\infty,a]}}U\right|_{\rho,0,1}^{2}
\end{align*}
for $U\in H_{\rho,1}(\mathbb{R};D(C_{0}^{\ast}C_{0})$) from which
\[
\left|\chi_{_{]-\infty,a]}}U\right|_{\rho,0,1}\leq\frac{2}{c_{1}}\;\left|\chi_{_{]-\infty,a]}}\left(\partial_{0}\eta_{0}+C_{0}^{*}C_{0}\right)U\right|_{\rho,0,-1}
\]
follows. The result then follows by continuous extension.\end{proof}

Note that the estimate obtained here is a slightly stronger causality
estimate than available in the general case of Theorem \ref{-Solution-0}.

\begin{rem}Of course we also have (since $\left|\phi\right|_{\rho,0,0}\leq\left|\phi\right|_{\rho,0,1}$
for $\phi\in H_{\rho,0}\left(\mathbb{R},D\left(C_{0}\right)\right)$)

\[
\left|SF\right|_{D\left(\overline{\partial_{0}\eta_{0}+C_{0}^{*}C_{0}}\right)}\leq\sqrt{1+C_{1}^{2}}\:\left|F\right|_{\rho,0,0}.
\]
\end{rem}

We also note the resulting energy balance law for solutions of \eqref{eq:basic-eq}.

\begin{thm}\label{thm:(Energy-balance-law)}(Energy balance law)
For a right-hand side $F\in H_{\rho,1}\left(\mathbb{R},H_{0}\right)$
with $F=0$ on $\left[T_{0},T_{1}\right]$ we have for the solution
$U\in H_{\rho,1}\left(\mathbb{R},D\left(C_{0}\right)\right)$ 
\begin{eqnarray*}
 &  & \frac{1}{2}\left\langle U|\eta U\right\rangle _{H}\left(T_{1}\right)+\int_{\left[T_{0},T_{1}\right]}\left\langle CU|CU\right\rangle _{H}=\\
 &  & =\frac{1}{2}\left\langle U|\eta U\right\rangle _{H}\left(T_{0}\right).
\end{eqnarray*}
\end{thm}

\begin{proof}For $F=0$ on $\left[T_{0},T_{1}\right]$ we have
\begin{align*}
0 & =\int_{\left[T_{0},T_{1}\right]}\left\langle U|\partial_{0}\eta U\right\rangle _{H}+\int_{\left[T_{0},T_{1}\right]}\left\langle CU|CU\right\rangle _{H}\\
 & =\frac{1}{2}\left\langle U|\eta U\right\rangle _{H}\left(T_{1}\right)-\frac{1}{2}\left\langle U|\eta U\right\rangle _{H}\left(T_{0}\right)+\\
 & +\int_{\left[T_{0},T_{1}\right]}\left\langle CU|CU\right\rangle _{H},
\end{align*}
where we have used the Sobolev embedding theorem to justify the integration
by parts. Furthermore note that the time-derivative commutes with
the solution operator. \end{proof}

For later purposes we analyse the underlying Hilbert spaces 
\begin{eqnarray*}
H & = & H_{0}\oplus H_{0}^{\perp}\\
H_{0} & = & \left(N\left(C\right)\cap N\left(\eta\right)\right)^{\perp}
\end{eqnarray*}
further. \textcolor{black}{For a Hilbert space $K$ and a subspace $L\subseteq K$, we define 
\[
   K\ominus L\coloneqq K \cap L^{\bot}.
\]}

\begin{lem}\label{lem:For-later-purposes}We have
\begin{eqnarray*}
H_{0} & = & R\left(C^{*}\right)\oplus\left(N\left(C\right)\cap H_{0}\right)\\
 & = & R\left(C^{*}\right)\oplus\left(N\left(C\right)\cap\overline{R\left(\eta\right)}\right)\oplus\left(\left(N\left(C\right)\cap H_{0}\right)\ominus\left(N\left(C\right)\cap\overline{R\left(\eta\right)}\right)\right).
\end{eqnarray*}
\end{lem}\begin{proof}By the projection theorem we have 
\[
H=R(C^{\ast})\oplus N(C).
\]
Intersecting both sides with $H_{0}$ and using that $R(C^{\ast})=N(C){}^{\bot}\subseteq H_{0}$
we obtain the first decomposition. For the second one, we observe
that $N(C)\cap\overline{R\left(\eta\right)}$ is a closed subspace
of $N(C)\cap H_{0},$ since $\overline{R\left(\eta\right)}=N(\eta)^{\bot}\subseteq H_{0}$
and hence, by the projection theorem 
\[
N(C)\cap H_{0}=\left(N(C)\cap\overline{R\left(\eta\right)}\right)\oplus\left((N(C)\cap H_{0})\ominus\left(N(C)\cap\overline{R\left(\eta\right)}\right)\right),
\]
which gives the second decomposition.\end{proof}

\begin{example} As a more elaborate illustrational example let us
consider the solution to the linear part of the so-called ``bidomain
model''\footnote{We are indebted to Ralph Chill for drawing our attention to this model.}
used in cardiac electrophysiology, see \cite{Bourgault2009458}. For
this let $\Omega\subseteq\mathbb{R}^{d}$ be open, bounded and connected
satisfying the segment property. The equation in question is given
by 
\[
\left(\partial_{0}\left(\begin{array}{cc}
1 & 1\\
1 & 1
\end{array}\right)+C^{*}C\right)U=F
\]
with some given data $F$ taking values in the state space 
\[
H=L^{2}\left(\Omega\right)\oplus L^{2}\left(\Omega\right)
\]
and
\[
C=\left(\begin{array}{cc}
\sqrt{\sigma_{1}} & 0\\
0 & \sqrt{\sigma_{2}}
\end{array}\right)\left(\begin{array}{cc}
\grad & 0\\
0 & \grad
\end{array}\right)
\]
with $\sigma_{k}\in L(L^{2}(\Omega,\mathbb{R}^{d}))$, $k\in\{1,2\}$,
selfadjoint and strictly positive definite with $D(C)=H^{1}(\Omega)\oplus H^{1}(\Omega)$
and $X=L^{2}(\Omega)^{d}\oplus L^{2}(\Omega)^{d}$ as well as 
\[
\eta=\left(\begin{array}{cc}
1 & 1\\
1 & 1
\end{array}\right).
\]
Note that $\grad$ (and therefore also $C$) has closed range, as
a standard contradiction argument using the compactness of the embedding
$H^{1}(\Omega)\hookrightarrow L^{2}(\Omega)$ eventually proving a
Poincare-type estimate shows; in fact we have
\begin{equation}
\left|u\right|_{L^{2}(\Omega)}\leq k\left|\grad u\right|_{L^{2}(\Omega,\mathbb{R}^{d})}\label{eq:pe}
\end{equation}
for all $u\in D(\grad)$ with $\int_{\Omega}u=0$ and some $k\geq0$.

Next, we aim at applying our abstract findings. In particular, we
need to establish the estimate in \eqref{eq:assume}. For this, let
us describe the reduced state space, $H_{0}$, first. We have 
\begin{eqnarray*}
H_{0} & = & \left(N\left(\eta\right)\cap N\left(C\right)\right)^{\perp}\\
 & = & \left\{ V\in L^{2}(\Omega)\oplus L^{2}(\Omega)\big|V=\left(\begin{array}{c}
u\\
-u
\end{array}\right)\text{ for some }u\in N(\grad)\right\} ^{\perp}\\
 & = & \left\{ V\in L^{2}(\Omega)\oplus L^{2}(\Omega)\big|V=\alpha\left(\begin{array}{c}
\chi_{\Omega}\\
-\chi_{\Omega}
\end{array}\right)\text{ for some }\alpha\in\mathbb{R}\right\} ^{\perp}\\
 & = & \left\{ (W_{1},W_{2})\in L^{2}(\Omega)\oplus L^{2}(\Omega)\big|\int_{\Omega}W_{1}(x) dx=\int_{\Omega}W_{2}(x)dx\right\} ,
\end{eqnarray*}

where in the second last equality we have used that $\Omega$ is connected
in order to have that $N(\grad)=\textcolor{black}{\lspan}\chi_{\Omega}$. According to
our abstract theory we need an estimate of the form
\[
\left|P_{R\left(\eta\right)}U\right|_{H}^{2}+\left|C_{0}U\right|_{X}^{2}\geq c_{*}\left|U\right|_{H}^{2}.
\]
holding for all 
\[
U\in D\left(C_{0}\right)\subseteq H_{0}=R\left(C^{*}\right)\oplus\left(N\left(C\right)\cap R\left(\eta\right)\right)\oplus\left(\left(N\left(C\right)\cap H_{0}\right)\ominus\left(N\left(C\right)\cap R\left(\eta\right)\right)\right)
\]
for some $c_{*}>0$ and where $P_{R(\eta)}$ denotes the projection
onto the range $R\left(\eta\right)=\overline{R\left(\eta\right)}$
of $\eta$. Take $U=U_{0}+U_{1}+U_{2}$ in the sense of this orthogonal
decomposition. First we note that 
\begin{eqnarray*}
N\left(C\right)\cap R\left(\eta\right) & = & \left\{ \alpha\left(\begin{array}{c}
\chi_{_{\Omega}}\\
\chi_{_{\Omega}}
\end{array}\right)\Big|\alpha\in\mathbb{R}\right\} 
\end{eqnarray*}
and
\begin{eqnarray*}
N\left(C\right)\cap H_{0} & = & \left\{ \alpha\left(\begin{array}{c}
\chi_{_{\Omega}}\\
\chi_{_{\Omega}}
\end{array}\right)\Big|\alpha\in\mathbb{R}\right\} ,
\end{eqnarray*}
so 
\[
U_{2}=0.
\]
Thus, we infer that
\[
H_{0}=R\left(C^{*}\right)\oplus\left(N\left(C\right)\cap R\left(\eta\right)\right).
\]
Moreover, by \eqref{eq:pe} and the assumptions on $\sigma_{k}$,
we find $c>0$ satisfying
\[
c\left|U_{0}\right|_{H}\leq\left|CU_{0}\right|_{H}\quad(U_{0}\in R(C^{*})\cap D(C)).
\]
Hence, for all $U\in D(C_{0})$ we have with $U=U_{0}+U_{1}$ for
uniquely determined $U_{0}\in R(C^{*})\cap D(C)$ and $U_{1}\in N(C)\cap R(\eta)$
that 
\begin{eqnarray*}
\left|P_{R\left(\eta\right)}U\right|_{H}^{2}+\left|C_{0}U\right|_{X}^{2} & = & \left|P_{R\left(\eta\right)}U_{0}+U_{1}\right|_{H}^{2}+\left|C_{0}U_{0}\right|_{X}^{2}\\
 & = & \left|P_{R\left(\eta\right)}U_{0}\right|_{H}^{2}+\left|U_{1}\right|_{H}^{2}+2\langle P_{R(\eta)}U_{0},U_{1}\rangle_{H}+\left|C_{0}U_{0}\right|_{X}^{2}\\
 & = & \left|P_{R\left(\eta\right)}U_{0}\right|_{H}^{2}+\left|U_{1}\right|_{H}^{2}+2\langle U_{0},U_{1}\rangle_{H}+\left|C_{0}U_{0}\right|_{X}^{2}\\
 & = & \left|P_{R\left(\eta\right)}U_{0}\right|_{H}^{2}+\left|U_{1}\right|_{H}^{2}+\left|C_{0}U_{0}\right|_{X}^{2}\\
 & \geq & \left|U_{1}\right|_{H}^{2}+c^{2}\left|U_{0}\right|_{H}^{2}.
\end{eqnarray*}
Thus, we found as desired
\[
\left|P_{R\left(\eta\right)}U\right|_{H}^{2}+\left|C_{0}U\right|_{X}^{2}\geq\min\left\{ 1,c^{2}\right\} \left|U\right|_{H}^{2}.
\]
Thus, well-posedness of the evo-system is implied by Theorem \ref{thm:Solution}.
Moreover, since $\eta\left[R\left(C^{*}\right)\right]\subseteq R\left(C^{*}\right)$
the problem can be further reduced to an evo-system in the subspace
$R\left(C^{*}\right)$ and an ordinary differential equation in $N\left(C\right)\cap R\left(C\right)$.
This insight might be useful, when dealing with problems in the light
of homogenisation, see e.g. \cite[Theorem 4.7]{Wau13} for this.\end{example}

\section{\label{sec:Application-to-a}Application to a Degenerate Evo-System
Associated with the Eddy Current Problem}

In this section, we shall now turn to our main application. Consider
the system
\begin{equation}
\begin{array}{rcl}
\sigma\mathrm{E}-\curl\mathrm{H} & = & -\mathrm{J}\\
\partial_{0}\mu\mathrm{H}+\interior\curl\mathrm{E} & = & K
\end{array}\label{eq:eddy-system}
\end{equation}
in an arbitrary non-empty open bounded set $\Omega\subseteq\mathbb{R}^{3}$
with connected boundary. We will require more regularity properties
for $\Omega,$ in the following.

After having specified the constituents of this system of two equations,
we shall reformulate the system in order to be in a position to apply
our general well-posedness theorem. This reformulation will then be
studied and related to the system \eqref{eq:eddy-system}. We shall
show that the solution for the reformulation yields a solution for
the equation, we started out with. \textcolor{black}{In view of the particular situation of the eddy current model at hand,} though this being a natural property
to ask for, it appears to have been overlooked in the literature so
far.

We specify the operators occurring in \eqref{eq:eddy-system} next.
The operator $\interior\curl$ denotes the closure of the classical
vector analytic operation $\curl$ defined on $C_{\infty}$-vector
fields with compact support in $\Omega$ considered as a mapping in
$L^{2}\left(\Omega,\mathbb{R}^{3}\right)$, that is,
\[
\interior\curl\colon D(\interior\curl)\subseteq L^{2}(\Omega,\mathbb{R}^{3})\to L^{2}(\Omega,\mathbb{R}^{3})
\]
given by
\begin{align*}
 & \phi\in D(\interior\curl),\psi=\interior\curl\phi\\
 & \iff\text{There exists a sequence }(\phi_{n})_{n}\text{ in }\interior{C}_{\infty}(\Omega,\mathbb{R}^{3})\text{ such that }\\
 & \quad\quad\quad\phi_{n}\to\phi\text{ and }\left(\begin{array}{c}
\partial_{2}\phi_{n,3}-\partial_{3}\phi_{n,2}\\
\partial_{3}\phi_{n,1}-\partial_{1}\phi_{n,3}\\
\partial_{1}\phi_{n,2}-\partial_{2}\phi_{n,1}
\end{array}\right)\to\psi\text{ in }L^{2}(\Omega,\mathbb{R}^{3})\text{ as }n\to\infty.
\end{align*}
\textcolor{black}{We emphasise that for smooth $\Omega$ belonging to the domain of $D(\interior{\curl})$ is equivalent to the (classical) vanishing of tangential component at the boundary.}
We define
\[
\curl\:\coloneqq\left(\interior\curl\right)^{*},
\]
which is the so-called weak $\curl$-derivative in $L^{2}\left(\Omega,\mathbb{R}^{3}\right)$.
The equations can now be written as a block operator matrix system
as 
\begin{equation}
\left(\partial_{0}\left(\begin{array}{cc}
0 & 0\\
0 & \mu
\end{array}\right)+\left(\begin{array}{cc}
\sigma & 0\\
0 & 0
\end{array}\right)+\left(\begin{array}{cc}
0 & -\curl\\
\interior\curl & 0
\end{array}\right)\right)\left(\begin{array}{c}
\mathrm{E}\\
\mathrm{H}
\end{array}\right)=\left(\begin{array}{c}
-\mathrm{J}\\
\mathrm{K}
\end{array}\right).\label{eq:pre-Max}
\end{equation}
\textcolor{black}{\begin{rem}
It might seem unphsyical to assume a non-zero source term $K$ on the right-hand side. In the formulation of evoluationary equations in particular concerning the reformulation of appropriate initial value problems as evoluationary equations with particular right-hand side it so happens that $K$ might be non-zero. We refer to Remark \ref{local-in-time}(b) and to \cite[Example 9.44]{STW20} for the details.
\end{rem}}
Furthermore, assume that 
\[
\mu:L^{2}\left(\Omega\right)\to L^{2}\left(\Omega\right)
\]
is selfadjoint and strictly positive definite. The assumption on $\sigma:L^{2}\left(\Omega\right)\to L^{2}\left(\Omega\right)$
is less standard. We shall assume a certain degree of degeneracy,
which is specified in the following assumption. For convenience of
the reader we denote the vector analytical operators defined on the
whole of $\Omega$ by $\curl$, $\grad$, and $\dive$ (and the respective
ones with full homogeneous boundary conditions by $\interior\curl$,
$\interior\grad,$ and $\interior\dive$). For operators defined on
other domains $\Omega_{c}$, we shall use this domain as an index
to refer to these operators such as for example $\grad_{\Omega_{c}}$
(the operator $\interior\grad_{\Omega_{c}}$ is the operator acting
as $\grad_{\Omega_{c}}$ with domain restricted to $H_{0}^{1}(\Omega_{c})$).

\begin{rem} As $\Omega$ is bounded, we have that $R(\interior\grad)$
is closed by Poincare's inequality. Moreover, $R(\interior\grad)\subseteq N(\interior\curl)$
and thus, the projection theorem gives 
\begin{align*}
N(\interior\curl) & =R(\interior\grad)\oplus\left(R(\interior\grad)^{\bot}\cap N(\interior\curl)\right)\\
 & =R(\interior\grad)\oplus\left(N(\dive)\cap N(\interior\curl)\right).
\end{align*}
The space 
\[
\mathcal{H}_{D,\Omega}\coloneqq N(\dive)\cap N(\interior\curl)
\]
is known as the space of harmonic Dirichlet fields. Since the boundary
of $\Omega$ is connected, it follows that $\mathcal{H}_{D,\Omega}=\{0\}$
by \cite[Theorem 1]{key-11} and thus, 
\begin{equation}
N(\interior\curl)=R(\interior\grad).\label{eq:N(curl)_is_grad}
\end{equation}

\end{rem}

\begin{hypo}\label{hyp:sigma} Let $\Omega_{c}\subseteq\Omega$ be
open. Moreover, assume that $\overline{\Omega}_{c}\subseteq\Omega$
and that $\Omega_{c}$ has a ($3$-dimensional) Lebesgue null set
as topological boundary and is such that $\Omega_{c}$ has finitely
many connected components and the connected components of $\Omega_{c}$
have disjoint closures. We also assume that $\Omega_{c}$ is such
that 
\begin{equation}
D(\grad_{\Omega_{c}})=\chi_{\Omega_{c}}\left[D(\interior\grad)\right].\label{eq:extension-prp}
\end{equation}
Let 
\[
\tilde{\sigma}:L^{2}\left(\Omega_{c},\mathbb{R}^{3}\right)\to L^{2}\left(\Omega_{c},\mathbb{R}^{3}\right)
\]
such that $\tilde{\sigma}$ is strictly positive definite. We shall
assume that $\sigma$ is degenerate in the sense that\footnote{In this case
\[
\chi_{_{\Omega_{c}}}:=\iota_{_{\Omega_{c}}}\iota_{_{\Omega_{c}}}^{*}
\]
is the orthogonal projector $P_{R(\sigma)}$ from $H=L^{2}\left(\Omega,\mathbb{R}^{3}\right)$
onto the closed linear subspace $R(\sigma)=\iota_{_{\Omega_{c}}}\left[L^{2}\left(\Omega_{c},\mathbb{R}^{3}\right)\right]$
(The canonical embedding $\iota_{_{\Omega_{c}}}$ of $L^{2}\left(\Omega_{c},\mathbb{R}^{3}\right)$
into $L^{2}\left(\Omega,\mathbb{R}^{3}\right)$ is via ``extension
by zero'').}
\[
\sigma=\iota_{\Omega_{c}}\tilde{\sigma}\:\iota_{\Omega_{c}}^{*}.
\]

\end{hypo}

We note here that \eqref{eq:extension-prp} indeed is a regularity
requirement for $\Omega_{c}$. In maybe more familiar terms, this
requirement equivalently reads as
\[
H^{1}(\Omega_{c})=\{\phi\in L^{2}(\Omega_{c})|\text{ there is }\tilde{\phi}\in H_{0}^{1}(\Omega)\text{ such that }\chi_{\Omega_{c}}\tilde{\phi}=\phi\}.
\]

\red{\begin{rem}
We comment some more on the condition \eqref{eq:extension-prp}. Since for every open set $\Omega_c\subseteq \mathbb{R}^3$, a function $u\in H^1_0(\Omega_c)$ if and only if 
\[
   \tilde u\coloneqq \begin{cases} u &\text{ on }\Omega_c\\
   0&\text{ else}\end{cases} \in H^1(\mathbb{R}^3),
\]the requirement \eqref{eq:extension-prp} is equivalent of $\Omega_c$ being an $H^1$-extension domain (see \cite{HKT2008} for the definition). The Calderon--Stein theorem asserts that Lipschitz domains are $H^1$-extension domains. An improvement of this result can be found in \cite{J1981}, which holds for so-called uniform domains allowing the boundary of $\Omega_c$ to have any Hausdorff dimension strictly less than $3$. A necessary criterion, however, is due to \cite[Theorem 2]{HKT2008} the \emph{measure density condition}, that is, there exists $c>0$ such that for all $x\in \Omega_c$ and $0<r\leq 1$ we have
\[
    \lambda(B(x,r)\cap \Omega_c)\geq cr^3,
\]where $\lambda(\cdot)$ denotes the Lebsgue measure in $\mathbb{R}^3$. Thus, all domains $\Omega_c$ failing this condition are no $H^1$-extension domains. Furthermore, if $\Omega_c$ has cracks of big enough Hausdorff-dimension (see e.g.~\cite{zbMATH06418162} for a two-dimensional setting), $\Omega_c$ is no $H^1$-extension domain.
\end{rem}}

We record an elementary consequence of the \textcolor{black}{assumptions} on $\sigma$.

\begin{prop}Assume Assumption \ref{hyp:sigma} to be in effect. Then
\begin{eqnarray*}
R\left(\sigma\right) & = & R\left(\chi_{_{\Omega_{c}}}\right)=L^{2}\left(\Omega_{c},\mathbb{R}^{3}\right),\\
N\left(\sigma\right) & = & R\left(1-\chi_{_{\Omega_{c}}}\right)=L^{2}\left(\Omega,\mathbb{R}^{3}\right)\ominus L^{2}\left(\Omega_{c},\mathbb{R}^{3}\right),\\
 & = & L^{2}\left(\Omega\setminus\overline{\Omega_{c}},\mathbb{R}^{3}\right),
\end{eqnarray*}
where $L^{2}\left(\Omega_{c},\mathbb{R}^{3}\right),\:L^{2}\left(\Omega\setminus\overline{\Omega_{c}},\mathbb{R}^{3}\right)$
are considered as subspaces of $L^{2}\left(\Omega,\mathbb{R}^{3}\right)$
via extension by zero.

\end{prop}

For the transcription of \eqref{eq:eddy-system} into a problem of
the form \eqref{eq:basic-eq}, we need to warrant the closed range
condition first. This, in turn, is a regularity requirement for $\Omega$:

\begin{hypo}\label{hyp:cr} Let $\Omega$ be such that $\interior\curl$
and consequently its adjoint $\curl$ have closed range:
\begin{equation}
R\left(\interior\curl\right),\:R\left(\curl\right)\text{ closed}.\label{eq:closed-range}
\end{equation}

\end{hypo}
\red{
\begin{rem}\label{rem:cr}
A closed range requirement is the fundamental property of linear equation
theory (see e.g.~\cite{TW14} for a corresponding result in elliptic theory) and linear operator equations with an operator having closed
range are therefore, since the beginning of last century also referred
to as \emph{normally solvable}. That for exterior domains or for $\mathbb{R}^{3}$
the differential operators (without or with associated homogeneous
boundary condition) $\grad$, $\interior\grad$, $\curl$, $\interior\curl$,
$\dive$, $\interior\dive$ have no closed range in an $L^{2}$-setting,
can be shown by approximations of the regularised fundamental solution
of the scalar or vectorial Laplacian. Note that for $\grad$, $\interior\grad$,
the closedness of the range is equivalent (this equivalence is due to the closed graph theorem, see e.g.~\cite[Remarks 3.2]{TW14}) to the Poincar\'e inequality,
which holds for $\Omega$ of bounded width, in particular for pipes
and slabs, where Rellich's selection theorem fails. For $R\left(\curl\right)$
and $R\left(\interior\curl\right)$ closedness has so far only
been obtained via a compact embedding result. Open subsets of Riemannian
manifold allowing for such a compact embedding result have been described
in \cite{Picard2001}, asking for $\Omega$ to satisfy only rather mild conditions (e.g., strictly weaker than $C^{1,1}$-domains and particularly not allowing for Gaffney's inequality to hold). We shall particularly refer to \cite{Bauer2016} for other boundary conditions.
\end{rem}
}

For later use, we shall further record the last two remaining regularity
properties needed for our well-posedness theorem to apply.

\begin{hypo}\label{hyp:rint} Assume \textcolor{black}{Assumption} \ref{hyp:sigma}
to be in effect. We shall assume that
\begin{eqnarray*}
N(\interior\curl)\cap L^{2}(\Omega\setminus\overline{\Omega}_{c},\mathbb{R}^{3}) & = & N(\interior\curl_{\Omega\setminus\overline{\Omega}_{c}})\text{ and }\\
N(\interior\curl)\cap L^{2}(\Omega_{c},\mathbb{R}^{3}) & = & N(\interior\curl_{\Omega_{c}}).
\end{eqnarray*}
Moreover, we suppose that
\[
R(\grad_{\Omega_{c}})\text{ is closed.}
\]

\end{hypo}

\begin{rem} Assumption \ref{hyp:rint} is another (boundary) regularity
property. For this to confirm, we realise that any $\phi\in D(\interior\curl_{\Omega\setminus\overline{\Omega}_{c}})$
extended by zero to the whole of $\Omega$ satisfies $\phi\in D(\interior\curl).$
Thus, in this sense, $N(\interior\curl)\cap L^{2}(\Omega\setminus\overline{\Omega}_{c},\mathbb{R}^{3})\supseteq N(\interior\curl_{\Omega\setminus\overline{\Omega}_{c}}).$
For the other inclusion the equality
\[
D(\interior\curl)\cap L^{2}(\Omega\setminus\overline{\Omega}_{c},\mathbb{R}^{3})=\{\phi\in D(\interior\curl)|\phi=0\text{ on }\overline{\Omega}_{c}\}=D(\interior\curl_{\Omega\setminus\overline{\Omega}_{c}})
\]
is sufficient. If for instance, $\Omega\setminus\overline{\Omega}_{c}$
satisfies the segment property, the desired equality holds. The second
equation and the third property in the assumptions are fulfilled, if,
for instance, $\Omega_{c}$ has the segment property. \red{We refer to Remark \ref{rem:cr} for the limitations of the closed range requirement.}

\end{rem}

We are now in the position to state the setting for the application
of Theorem \ref{thm:Solution}.

We put
\begin{align}
H=X & =L^{2}\left(\Omega,\mathbb{R}^{3}\right),\nonumber \\
C\colon D(C)\subseteq H & \to X,\nonumber \\
E & \mapsto\mu^{-1/2}\interior\curl E,\label{eq:setting}\\
D(C) & =D(\interior\curl),\nonumber \\
\eta & =\sigma.\nonumber 
\end{align}

\begin{prop}\label{prop:setass} Let $\Omega\subseteq\mathbb{R}^{3}$
be open with connected boundary. Assume Assumption \ref{hyp:sigma},
\ref{hyp:cr}, \ref{hyp:rint},  to be in effect. Then $C$ and $\eta$
as given in \eqref{eq:setting} satisfy the assumptions in Theorem
\ref{thm:Solution}.

\end{prop}

The proof of Proposition \ref{prop:setass} requires a lot of preparations.
The main issue is of course to prove inequality \eqref{eq:assume}
under the current assumptions. Indeed, note that since $\mu$ is selfadjoint
and a topological isomorphism, we easily realise that $C$ is densely
defined and closed. Moreover, we obtain $C^{*}=\curl\mu^{-1/2}$ from
which we read off that
\[
R(C^{*})=R(\curl),
\]
which is assumed to be closed by Assumption \ref{hyp:cr}. Thus, we
are left with showing \eqref{eq:assume}. Before, however, doing so,
we reason, why it makes sense to look at the setting \eqref{eq:setting}
for solving \eqref{eq:eddy-system}.

\begin{rem}\label{rem:recovery} Using the assumptions of Proposition \ref{prop:setass}
and using the notation introduced in the previous section, we are
led to the evo-system
\[
\partial_{0}\eta_{0}u+C_{0}^{*}C_{0}u=-j\in H_{\rho,0}\left(\mathbb{R},H_{0}\right),
\]
with
\[
j\,\,:=\mathrm{J}-C_{0}^{*}\partial_{0}^{-1}\mu^{-1/2}\mathrm{K},
\]
where
\[
\mathrm{J}\in H_{\rho,0}\left(\mathbb{R},H_{0}\right).
\]
Hence, with
\[
\mathrm{E}\coloneqq\partial_{0}u
\]
and 
\[
\mathrm{H}\coloneqq-\partial_{0}^{-1}\mu^{-1/2}\left(C_{0}\mathrm{E}-\mu^{-1/2}\mathrm{K}\right)
\]
we recover
\begin{eqnarray*}
\sigma\mathrm{E}-C^{*}\mu^{1/2}\mathrm{H} & = & -\mathrm{J}\\
\partial_{0}\mu^{1/2}\mathrm{H}+C\mathrm{E} & = & \mu^{-1/2}\mathrm{K}
\end{eqnarray*}
or
\begin{eqnarray*}
\sigma\mathrm{E}-\curl\mathrm{H} & = & -\mathrm{J},\\
\partial_{0}\mu\mathrm{H}+\interior\curl\:\mathrm{E} & = & \mathrm{K},
\end{eqnarray*}
which is \eqref{eq:eddy-system}. Note that the argument just presented
is an incarnation of Proposition \ref{prop:just}, which in turn yields
the solvability of the system, we started out with.\end{rem}

To demonstrate \eqref{eq:assume} we first recall Lemma \ref{lem:For-later-purposes}.
In particular, we have
\begin{align}
H_{0} & =\left(N(C)\cap N(\eta)\right)^{\bot}\nonumber \\
 & =R(C^{*})\oplus H_{1}\oplus H_{2},\text{ where }\label{eq:ortho-H0}\\
H_{1} & =N(C)\cap R(\eta)\text{ and }\nonumber \\
H_{2} & =\left(N(C)\cap H_{0}\right)\ominus\left(N(C)\cap R(\eta)\right).\nonumber 
\end{align}

In the following, we describe these spaces more explicitly. Throughout,
we shall assume that the assumptions of Proposition \ref{prop:setass}
are in effect. For the formulation of \red{Lemma \ref{lem:decom}}, we define for
an open set $\mathcal{O}\subseteq\mathbb{R}^{3}$
\[
\mathcal{H}_{D,\mathcal{O}}\coloneqq N(\dive_{\mathcal{O}})\cap N(\interior\curl_{\mathcal{O}}),
\]
the space of harmonic Dirichlet fields in $\mathcal{O}$. In the following
we will use the projection theorem in different spaces. For the sake
of readability, we will use indices at the orthogonal complements
in order to clarify, in which space we take the orthogonal complement. 

\red{In order to illustrate the following findings, we recall one of the main results in \cite{key-11}, namely the computation of the dimension of the harmonic Dirichlet fields. For this let $\mathcal{O}\subseteq \mathbb{R}^3$ be open bounded with continuous boundary. We denote \[{\textsc{cc}}_{(b)}(\mathcal{O})\coloneqq \{ z \subseteq \mathbb{R}^3\setminus\overline{\mathcal{O}}; z\text{ (bounded) connected component}\}.\] For $z\in {\textsc{cc}}_{b}(\mathcal{O})$ let $\psi_z\in C_c^\infty(\mathbb{R}^3)$ such that
\[
   \psi_z (x)=\begin{cases} 1, & x\in z,\\
   0, & x \in \bigcup_{z'\in {\textsc{cc}}(\mathcal{O})\setminus \{z\}} z'.
   \end{cases}
\]Define $q_z \coloneqq \grad \psi_z$ and $\phi_z \coloneqq \pi_{\mathcal{H}_{D,\mathcal{O}}} q_z|_{\mathcal{O}}$, where  $\pi_{\mathcal{H}_{D,\mathcal{O}}}$ denotes the $L^2(\mathcal{O})^3$-orthogonal projection onto $\mathcal{H}_{D,\mathcal{O}}$.
\begin{thm}[{{{\cite[Theorem 1]{key-11}}}}]\label{thm:picardEdin} Assume $\Omega \subseteq \mathbb{R}^3$ to be bounded with continuous boundary. Let $m\in \mathbb{N}$ be the number of connected components of $\mathbb{R}^3\setminus\overline{\Omega}$. Then
\[
   \dim \mathcal{H}_{D,\Omega} = m-1.
\]
More precisely, $(\phi_z)_{z\in {\textsc{cc}}_{b}(\mathcal{O})}$ constitutes a basis for $ \mathcal{H}_{D,\Omega}$.
\end{thm}}

\red{
\begin{rem}(a) In the particular case that $\Omega$ is a ball and $\Omega_c$ is an inscribed ball, the number of connected components of $\mathbb{R}^3\setminus\big(\Omega\setminus\overline{\Omega_c}\big)$ is 2. Thus,
\[
  \dim \mathcal{H}_{D,\Omega\setminus\overline{\Omega_c}}=1.
\] It is possible to compute this function by appropriately projecting the gradient of a function, which is identically $1$ on $\Omega_{c}$ and $0$ outside $\Omega$. It is possible to compute such a solution numerically, by solving a variational problem. We refer to \cite{PaulyWaurick2020} for the details. In the situation of $\Omega_c$ being a ball, we also have that
\[
 \mathcal{H}_{D,\Omega_c}=\{0\}.
\]
(b) Note that the construction principle to obtain a basis for the space of harmonic Dirichlet fields extends to other differential operators. For this, we also refer to \cite{PaulyWaurick2020} for the details using the machinery of Hilbert complexes.
\end{rem}}

\begin{lem} \label{lem:decom}The following equalities hold:
\begin{align}
H_{0} & =\left(N\left(\dive_{\Omega\setminus\overline{\Omega_{c}}}\right)\cap\text{\ensuremath{\mathcal{H}}}_{D,\Omega\setminus\overline{\Omega_{c}}}^{\perp_{L^{2}(\Omega\setminus\overline{\Omega_{c}})}}\right)\oplus L^{2}\left(\Omega_{c},\mathbb{R}^{3}\right),\nonumber \\
H_{0}^{\bot_{L^{2}(\Omega)}} & =N(\interior\curl_{\Omega\setminus\overline{\Omega_{c}}})\nonumber \\
H_{1} & =N\left(\interior\curl_{\Omega_{c}}\right)\nonumber \\
 & =\left(N\left(\dive_{\Omega_{c}}\right)\cap\text{\ensuremath{\mathcal{H}}}_{D,\Omega_{c}}^{\perp_{L^{2}(\Omega_{c})}}\right)^{\perp_{L^{2}(\Omega_{c})}}\label{eq:ortho-H1}\\
H_{1}^{\bot_{L^{2}(\Omega)}} & =\left(N\left(\dive_{\Omega_{c}}\right)\cap\text{\ensuremath{\mathcal{H}}}_{D,\Omega_{c}}^{\perp_{L^{2}(\Omega_{c})}}\right)\oplus L^{2}\left(\Omega\setminus\overline{\Omega_{c}},\mathbb{R}^{3}\right)\nonumber \\
H_{2} & =N\left(\interior\curl\right)\cap\left(\left(N\left(\dive_{\Omega_{c}}\right)\cap\text{\ensuremath{\mathcal{H}}}_{D,\Omega_{c}}^{\perp_{L^{2}(\Omega_{c})}}\right)\oplus\left(N\left(\dive_{\Omega\setminus\overline{\Omega_{c}}}\right)\cap\text{\ensuremath{\mathcal{H}}}_{D,\Omega\setminus\overline{\Omega_{c}}}^{\perp_{L^{2}(\Omega\setminus\overline{\Omega_{c}})}}\right)\right)\nonumber 
\end{align}

\end{lem}

\begin{proof}Using Assumptions \ref{hyp:rint} and \ref{hyp:sigma},
we obtain
\begin{align*}
H_{0}^{\bot_{L^{2}(\Omega)}} & =N(C)\cap N(\eta)\\
 & =N(\interior\curl)\cap L^{2}(\Omega\setminus\overline{\Omega}_{c},\mathbb{R}^{3})\\
 & =N(\interior\curl_{\Omega\setminus\overline{\Omega}_{c}}).
\end{align*}
Since $\interior\grad_{\Omega\setminus\overline{\Omega}_{c}}^{*}=-\dive_{\Omega\setminus\overline{\Omega}_{c}}$
with adjoint computed in $L^{2}(\Omega\setminus\overline{\Omega}_{c},\mathbb{R}^{3})$
and $R(\interior\grad_{\Omega\setminus\overline{\Omega}_{c}})\subseteq N(\interior\curl_{\Omega\setminus\overline{\Omega}_{c}})$,
we thus obtain
\begin{eqnarray*}
H_{0} & = & N(\interior\curl_{\Omega\setminus\overline{\Omega_{c}}})^{\bot_{L^{2}(\Omega)}}\\
 & = & N(\interior\curl_{\Omega\setminus\overline{\Omega_{c}}})^{\bot_{L^{2}(\Omega\setminus\overline{\Omega_{c}})}}\oplus L^{2}(\Omega_{c},\mathbb{R}^{3})\\
 & = & \left(R(\interior\grad_{\Omega\setminus\overline{\Omega}_{c}})\oplus\mathcal{H}_{D,\Omega\setminus\overline{\Omega}_{c}}\right)^{\bot_{L^{2}(\Omega\setminus\overline{\Omega_{c}})}}\oplus L^{2}(\Omega_{c},\mathbb{R}^{3})\\
 & = & \left(N(\dive_{\Omega\setminus\overline{\Omega}_{c}})\cap\mathcal{H}_{D,\Omega\setminus\overline{\Omega}_{c}}^{\bot_{L^{2}(\Omega\setminus\overline{\Omega_{c}})}}\right)\oplus L^{2}(\Omega_{c},\mathbb{R}^{3}).
\end{eqnarray*}
Next, we have by Assumption \ref{hyp:rint}
\begin{eqnarray*}
H_{1} & = & N(C)\cap R(\eta)\\
 & = & N(\interior\curl)\cap L^{2}(\Omega_{c},\mathbb{R}^{3})\\
 & = & N(\interior\curl_{\Omega_{c}}).
\end{eqnarray*}
An analogous argument as already done for $H_{0}$ implies the asserted
equation for $H_{1}^{\bot_{L^{2}(\Omega)}}$, which in turn implies
the second expression for $H_{1}$. Finally, from $R(C^{*})=R(\curl)$
and the already derived expression for $H_{0}$, we deduce 
\[
N(C)\cap H_{0}=N\left(\interior\curl\right)\cap\left(\left(N\left(\dive_{\Omega\setminus\overline{\Omega_{c}}}\right)\cap\text{\ensuremath{\mathcal{H}}}_{D,\Omega\setminus\overline{\Omega_{c}}}^{\perp_{L^{2}(\Omega\setminus\overline{\Omega_{c}})}}\right)\oplus L^{2}\left(\Omega_{c},\mathbb{R}^{3}\right)\right)
\]
 and therefore
\begin{eqnarray*}
H_{2} & = & \left(N\left(C\right)\cap H_{0}\right)\ominus\left(N\left(C\right)\cap R\left(\eta\right)\right)\\
 & = & \left(N\left(C\right)\cap H_{0}\right)\cap H_{1}^{\bot_{L^{2}(\Omega)}},\\
 & = & N\left(\interior\curl\right)\cap\left(\left(N\left(\dive_{\Omega\setminus\overline{\Omega_{c}}}\right)\cap\text{\ensuremath{\mathcal{H}}}_{D,\Omega\setminus\overline{\Omega_{c}}}^{\perp_{L^{2}(\Omega\setminus\overline{\Omega_{c}})}}\right)\oplus L^{2}\left(\Omega_{c},\mathbb{R}^{3}\right)\right)\cap\\
 &  & \cap\left(\left(N\left(\dive_{\Omega_{c}}\right)\cap\text{\ensuremath{\mathcal{H}}}_{D,\Omega_{c}}^{\perp_{L^{2}(\Omega_{c})}}\right)\oplus L^{2}\left(\Omega\setminus\overline{\Omega_{c}},\mathbb{R}^{3}\right)\right),\\
 & = & N\left(\interior\curl\right)\cap\left(\left(N\left(\dive_{\Omega_{c}}\right)\cap\text{\ensuremath{\mathcal{H}}}_{D,\Omega_{c}}^{\perp_{L^{2}(\Omega_{c})}}\right)\oplus\left(N\left(\dive_{\Omega\setminus\overline{\Omega_{c}}}\right)\cap\text{\ensuremath{\mathcal{H}}}_{D,\Omega\setminus\overline{\Omega_{c}}}^{\perp_{L^{2}(\Omega\setminus\overline{\Omega_{c}})}}\right)\right).
\end{eqnarray*}
\end{proof}

A next step towards the desired inequality is provided next.

\begin{lem}\label{lem:For-later-purposes-1}We have for $U_{k}\in H_{k}$,
$k\in\{1,2\}$, 
\[
\left|\chi_{_{\Omega_{c}}}\left(U_{1}+U_{2}\right)\right|^{2}=\left|U_{1}\right|^{2}+\left|\chi_{_{\Omega_{c}}}U_{2}\right|^{2}.
\]
\end{lem}

\begin{proof}By Lemma \ref{lem:decom}, we obtain that
\[
\chi_{_{\Omega_{c}}}U_{2}\in N\left(\dive_{\Omega_{c}}\right)\cap\text{\ensuremath{\mathcal{H}}}_{D,\Omega_{c}}^{\perp_{L^{2}(\Omega_{c})}}
\]
Hence, with \eqref{eq:ortho-H1} we deduce
\begin{eqnarray*}
\left|\chi_{_{\Omega_{c}}}\left(U_{1}+U_{2}\right)\right|^{2} & = & \left|\chi_{_{\Omega_{c}}}U_{1}\right|^{2}+\left\langle \chi_{_{\Omega_{c}}}U_{1}|\chi_{_{\Omega_{c}}}U_{2}\right\rangle +\left|\chi_{_{\Omega_{c}}}U_{2}\right|^{2}\\
 & = & \left|U_{1}\right|^{2}+\left\langle U_{1}|\chi_{_{\Omega_{c}}}U_{2}\right\rangle +\left|\chi_{_{\Omega_{c}}}U_{2}\right|^{2}\\
 & = & \left|U_{1}\right|^{2}+\left|\chi_{_{\Omega_{c}}}U_{2}\right|^{2}.
\end{eqnarray*}
\end{proof}

By Assumption \ref{hyp:cr}, we deduce with an application of the
closed graph theorem, that there exists $k_{0}\geq0$ such that
\begin{equation}
\left|U\right|\leq k_{0}\left|\interior\curl U\right|\label{eq:intcu}
\end{equation}
for all $U\in D\left(\interior\curl\right)\cap R\left(C^{*}\right)$.
Finally we need a more subtle result, which is the key step towards
showing the desired inequality \eqref{eq:assume} in the present context.

\begin{prop}\label{prop:key}There exists $k_{1}\geq0$ so that 
\begin{equation}
\left|U\right|\leq k_{1}\left|\chi_{_{\Omega_{c}}}U\right|\label{eq:H2-estimate}
\end{equation}
for all $U\in H_{2}$ \end{prop}

In order to prove this proposition we need some preparations. We start
with the following observation.

\begin{lem}\label{lem:Hilbert_space_H3} Define 
\[
H_{3}\coloneqq\grad_{\Omega_{c}}\left[N(\dive_{\Omega_{c}}\grad_{\Omega_{c}})\cap\{\phi\,|\,\grad_{\Omega_{c}}\phi\in\mathcal{H}_{D,\Omega_{c}}^{\bot_{L^{2}(\Omega_{c})}}\}\right]\subseteq L^{2}(\Omega_{c})^{3}.
\]
Then $H_{3}$ is a closed subspace of $L^{2}(\Omega_{c})^{3}$ and
for $U\in H_{2}$ we have that $\chi_{\Omega_{c}}U\in H_{3}.$

\end{lem}
\begin{proof}
Obviously, $H_{3}$ is a subspace of $L^{2}(\Omega_{c})^{3}.$ For
proving the closedness of $H_{3},$ let $(\phi_{n})_{n\in\mathbb{N}}$
be a sequence in $N(\dive_{\Omega_{c}}\grad_{\Omega_{c}})\cap\left\{ \phi\,|\,\grad_{\Omega_{c}}\phi\in\mathcal{H}_{D,\Omega_{c}}^{\bot_{L^{2}(\Omega_{c})}}\right\} $
such that $\grad_{\Omega_{c}}\phi_{n}\to u$ for some $u\in L^{2}(\Omega_{c})^{3}.$
Since $R(\grad_{\Omega_{c}})$ is closed by Assumption \ref{hyp:rint}
we infer that $u=\grad_{\Omega_{c}}\phi$ for some $\phi\in D(\grad_{\Omega_{c}}).$
Since $\grad_{\Omega_{c}}\phi_{n}\in N(\dive_{\Omega_{c}})$ for each
$n\in\mathbb{N}$ it follows by the closedness of $N(\dive_{\Omega_c})$
that also $u=\grad_{\Omega_{c}}\phi\in N(\dive_{\Omega_{c}}),$ i.e.
$\phi\in N(\dive_{\Omega_{c}}\grad_{\Omega_{c}}).$ Finally, since
$\mathcal{H}_{D,\Omega_{c}}^{\bot}$ is closed and $\grad_{\Omega_{c}}\phi_{n}\in\mathcal{H}_{D,\Omega_{c}}^{\bot_{L^{2}(\Omega_{c})}}$
for each $n\in\mathbb{N}$, the same holds true for $u=\grad_{\Omega_{c}}\phi.$
Summarising, we have shown that $u\in H_{3}$ and thus, $H_{3}$ is
closed. \\
Take now $U\in H_{2}.$ In particular, $U\in N(\interior\curl)=R(\interior\grad)$
by (\ref{eq:N(curl)_is_grad}), and hence, $U=\interior\grad\,\psi$
for some $\psi\in D(\interior\grad).$ By Assumption \ref{hyp:sigma}
it follows that $\phi\coloneqq\chi_{\Omega_{c}}\psi\in D(\grad_{\Omega_{c}})$
and 
\[
\grad_{\Omega_{c}}\phi=\chi_{\Omega_{c}}\interior\grad\psi=\chi_{\Omega_{c}}U.
\]
Moreover, since $U\in H_{2}$, it follows by Lemma \ref{lem:decom}
that $\grad_{\Omega_{c}}\phi=\chi_{\Omega_{c}}U\in N(\dive_{\Omega_{c}})\cap\mathcal{H}_{D,\Omega_{c}}^{\bot_{L^{2}(\Omega_{c})}}$
which shows that $\chi_{\Omega_{c}}U\in H_{3}.$
\end{proof}
In the following, we consider the operator 
\begin{align*}
Z:H_{2} & \to H_{3}\\
U & \mapsto\chi_{\Omega_{c}}U.
\end{align*}
\begin{lem}\label{lem:Z_injective} The operator $Z$ is one-to-one.

\end{lem}
\begin{proof}
Let $U\in H_{2}$ with $ZU=\chi_{\Omega_{c}}U=0.$ Since $U=0$ on
$\Omega_{c}$ and $U\in N(\interior\curl),$ we infer by Assumption
\ref{hyp:rint} that $U\in N(\interior\curl_{\Omega\setminus\overline{\Omega_{c}}})$.
Moreover, by the definition of $H_{2}$ we get that $U\in N(\dive_{\Omega\setminus\overline{\Omega_{c}}})\cap\mathcal{H}_{D,\Omega\setminus\overline{\Omega_{c}}}^{\bot_{L^{2}(\Omega\setminus\overline{\Omega_{c}})}}$
and thus, 
\[
U\in N(\interior\curl_{\Omega\setminus\overline{\Omega_{c}}})\cap N(\dive_{\Omega\setminus\overline{\Omega_{c}}})\cap\mathcal{H}_{D,\Omega\setminus\overline{\Omega_{c}}}^{\bot_{L^{2}(\Omega\setminus\overline{\Omega_{c}})}}=\mathcal{H}_{D,\Omega\setminus\overline{\Omega_{c}}}\cap\mathcal{H}_{D,\Omega\setminus\overline{\Omega_{c}}}^{\bot_{L^{2}(\Omega\setminus\overline{\Omega_{c}})}}=\{0\}.\tag*{{\qedhere}}
\]
\end{proof}
\begin{lem} \label{lem:Z_onto} The operator $Z$ is onto.

\end{lem}
\begin{proof}
Let $W\in H_{3},$ i.e. $W\in N(\dive_{\Omega_{c}})\cap\mathcal{H}_{D,\Omega_{c}}^{\bot_{L^{2}(\Omega_{c})}}$
and there is $\phi\in D(\grad_{\Omega_{c}})$ with
\[
W=\grad_{\Omega_{c}}\phi.
\]
By (\ref{eq:extension-prp}) there is $\psi\in D(\interior\grad)$
such that $\phi=\chi_{\Omega_{c}}\psi$. Note that by Poincare's inequality,
$R(\interior\grad_{\Omega\setminus\overline{\Omega_{c}}})$ is a closed
subspace of $L^{2}(\text{\ensuremath{\Omega\setminus}\ensuremath{\overline{\Omega_{c}}})}$.
Denoting the orthogonal projector onto $R(\interior\grad_{\Omega\setminus\overline{\Omega_{c}}})$
by $P_{R(\interior\grad_{\Omega\setminus\overline{\Omega_{c}}})}$
we consider 
\[
-P_{R(\interior\grad_{\Omega\setminus\overline{\Omega_{c}}})}\chi_{\Omega\setminus\overline{\Omega_{c}}}\interior\grad\psi\in R(\interior\grad_{\Omega\setminus\overline{\Omega_{c}}}),
\]
and thus, we find $\theta\in D(\interior\grad_{\Omega\setminus\overline{\Omega_{c}}})$
with 
\[
\interior\grad_{\Omega\setminus\overline{\Omega_{c}}}\theta=-P_{R(\interior\grad_{\Omega\setminus\overline{\Omega_{c}}})}\chi_{\Omega\setminus\overline{\Omega_{c}}}\interior\grad\psi.
\]
We set 
\[
\tilde{\psi}\coloneqq\psi+\theta\in D(\interior\grad)
\]
and obtain 
\begin{align}
\chi_{\Omega\setminus\overline{\Omega_{c}}}\interior\grad\tilde{\psi} & =\chi_{\Omega\setminus\overline{\Omega_{c}}}\interior\grad\psi+\interior\grad_{\Omega\setminus\overline{\Omega_{c}}}\theta\nonumber \\
 & =(1-P_{R(\interior\grad_{\Omega\setminus\overline{\Omega_{c}}})})\chi_{\Omega\setminus\overline{\Omega_{c}}}\interior\grad\psi\nonumber \\
 & \in R(\interior\grad_{\Omega\setminus\overline{\Omega_{c}}})^{\bot_{L^{2}(\Omega\setminus\overline{\Omega_{c}})}}=N(\dive_{\Omega\setminus\overline{\Omega_{c}}}).\label{eq:psi-tilde-harmonic}
\end{align}
Finally, we note that 
\[
\mathcal{H}_{D,\Omega\setminus\overline{\Omega_{c}}}\subseteq N(\interior\curl_{\Omega\setminus\overline{\Omega_{c}}})=N(\interior\curl)\cap L^{2}(\Omega\setminus\overline{\Omega_{c}})=R(\interior\grad)\cap L^{2}(\Omega\setminus\overline{\Omega_{c}})\subseteq R(\interior\grad),
\]
where we have used Assumption \ref{hyp:rint} for the first equality
and (\ref{eq:N(curl)_is_grad}) for the second equality. Hence, $\mathcal{H}_{D,\Omega\setminus\Omega_{c}}$
is a closed subspace of $R(\interior\grad).$ We now define 
\[
U\coloneqq P_{\mathcal{H}_{D,\Omega\setminus\overline{\Omega_{c}}}^{\bot_{R(\interior\grad)}}}\interior\grad\tilde{\psi}
\]
and obtain 
\[
U\in\mathcal{H}_{D,\Omega\setminus\overline{\Omega_{c}}}^{\bot_{R(\interior\grad)}}=\mathcal{H}_{D,\Omega\setminus\overline{\Omega_{c}}}^{\bot_{L^{2}(\Omega)}}\cap R(\interior\grad)=\left(\mathcal{H}_{D,\Omega\setminus\overline{\Omega_{c}}}^{\bot_{L^{2}(\Omega\setminus\overline{\Omega_{c}})}}\oplus L^{2}(\Omega_{c})\right)\cap R(\interior\grad).
\]
Thus, in particular, $U\in N(\interior\curl)$ and $\chi_{\Omega\setminus\overline{\Omega_{c}}}U\in\mathcal{H}_{D,\Omega\setminus\overline{\Omega_{c}}}^{\bot_{L^{2}(\Omega\setminus\overline{\Omega_{c}})}}$.
Moreover, 
\[
U-\interior\grad\tilde{\psi}\in\mathcal{H}_{D,\Omega\setminus\overline{\Omega_{c}}}\subseteq N(\dive_{\Omega\setminus\Omega_{c}})
\]
and thus, in particular $U-\interior\grad\tilde{\psi}=0$ on $L^{2}(\Omega_{c})$
and 
\[
\chi_{\Omega\setminus\overline{\Omega_{c}}}U=\chi_{\Omega\setminus\overline{\Omega_{c}}}(U-\interior\grad\tilde{\psi})+\chi_{\Omega\setminus\overline{\Omega_{c}}}\interior\grad\tilde{\psi}\in N(\dive_{\Omega\setminus\overline{\Omega_{c}}}),
\]
where we have used (\ref{eq:psi-tilde-harmonic}). On the other hand,
we have 
\begin{align*}
\chi_{\Omega_{c}}U & =\chi_{\Omega_{c}}\interior\grad\tilde{\psi}\\
 & =\chi_{\Omega_{c}}\interior\grad\psi+\chi_{\Omega_{c}}\interior\grad\theta\\
 & =\grad_{\Omega_{c}}\chi_{\Omega_{c}}\psi\\
 & =\grad_{\Omega_{c}}\phi\\
 & =W\in N(\dive_{\Omega_{c}})\cap\mathcal{H}_{D,\Omega_{c}}^{\bot_{L^{2}(\Omega_{c})}},
\end{align*}
and thus, $U\in H_{2}$ with $ZU=W.$ This completes the proof. 
\end{proof}
Now we are able to prove Proposition \ref{prop:key}.
\begin{proof}[Proof of Proposition \ref{prop:key}]
 Since $Z:H_{2}\to H_{3}$ is continuous, one-to-one and onto, it
follows that $Z^{-1}:H_{3}\to H_{2}$ is continuous as well by the
closed graph theorem. Thus, the assertion follows with $k_{1}\coloneqq\|Z^{-1}\|.$ 
\end{proof}

We are finally in the position to prove inequality \eqref{eq:assume}
and, therefore, to complete the proof of Proposition \ref{prop:setass}.

\begin{lem}There is a positive constant $c_{0}$ such that we have
\begin{equation}
c_{0}\left|U\right|^{2}\leq\left|\sigma^{1/2}U\right|^{2}+\left|\interior\curl U\right|^{2}\label{eq:pos-def-eddy}
\end{equation}
for all $U\in D\left(\interior\curl\right)\cap H_{0}$. \end{lem}\begin{proof}By
the positive definiteness of $\tilde{\sigma}$, see Assumption \ref{hyp:sigma},
we obtain for all $U\in D(\interior\curl)\cap H_{0}$
\[
c_{*}\left|\chi_{_{\Omega_{c}}}U\right|^{2}+\left|\interior\curl U\right|^{2}\leq\left|\sigma^{1/2}U\right|^{2}+\left|\interior\curl U\right|^{2}
\]
for some $c_{*}>0$. Thus, the desired estimate follows if we can
show that there is $c>0$ such that for all $U\in D(\interior\curl)\cap H_{0}$
\[
c\left|U\right|^{2}\leq\left|\chi_{_{\Omega_{c}}}U\right|^{2}+\left|\interior\curl U\right|^{2}.
\]
We shall employ the above decomposition \eqref{eq:ortho-H0} so that
$U=U_{0}+U_{1}+U_{2}$ with $U_{0}\in R\left(\curl\mu^{-\frac{1}{2}}\right)$,
$U_{k}\in H_{k}$, $k\in\{1,2\}$. We compute using \eqref{eq:intcu},
Lemma \ref{prop:key}, and Lemma \ref{lem:For-later-purposes-1}
\begin{eqnarray*}
\left|U\right|^{2} & = & \left|U_{0}\right|^{2}+\left|U_{1}\right|^{2}+\left|U_{2}\right|^{2}\\
 & \leq & k_{0}^{2}\left|\interior\curl U_{0}\right|^{2}+k_{1}^{2}\left|\chi_{_{\Omega_{c}}}U_{2}\right|^{2}+\left|U_{1}\right|^{2}\\
 & \leq & k_{0}^{2}\left|\interior\curl U_{0}\right|^{2}+\max\left\{ 1,k_{1}^{2}\right\} \left|\chi_{_{\Omega_{c}}}\left(U_{1}+U_{2}\right)\right|^{2}\\
 & \leq & k_{0}^{2}\left|\interior\curl U_{0}\right|^{2}+2\max\left\{ 1,k_{1}^{2}\right\} \left|\chi_{_{\Omega_{c}}}\left(U_{0}+U_{1}+U_{2}\right)\right|^{2}+\\
 &  & +2\max\left\{ 1,k_{1}^{2}\right\} \left|\chi_{_{\Omega_{c}}}U_{0}\right|^{2},\\
 & \leq & k_{0}^{2}\left|\interior\curl U_{0}\right|^{2}+2\max\left\{ 1,k_{1}^{2}\right\} \left|\chi_{_{\Omega_{c}}}\left(U_{0}+U_{1}+U_{2}\right)\right|^{2}+\\
 &  & +2\max\left\{ 1,k_{1}^{2}\right\} \left|U_{0}\right|^{2},\\
 & \leq & k_{0}^{2}\left(1+2\max\left\{ 1,k_{1}^{2}\right\} \right)\left|\interior\curl U_{0}\right|^{2}+\\
 &  & +2\max\left\{ 1,k_{1}^{2}\right\} \left|\chi_{_{\Omega_{c}}}\left(U_{0}+U_{1}+U_{2}\right)\right|^{2},\\
 & \leq & \max\left\{ 2,2k_{1}^{2},k_{0}^{2}\left(1+2\max\left\{ 1,k_{1}^{2}\right\} \right)\right\} \left(\left|\interior\curl U\right|^{2}+\left|\chi_{_{\Omega_{c}}}U\right|^{2}\right)
\end{eqnarray*}
Thus we see that the estimate \eqref{eq:pos-def-eddy} holds for 
\[
c_{0}=\min\left\{ 1,c_{*}\right\} 
\]
with
\[
c_{*}=\frac{1}{\max\left\{ 2,2k_{1}^{2},k_{0}^{2}\left(1+2\max\left\{ 1,k_{1}^{2}\right\} \right)\right\} }.
\]
\end{proof}

We shall now summarise the findings of this section.

\begin{thm}\label{thm:mr}Let $\Omega\subseteq\mathbb{R}^{3}$ be
open with connected boundary. Assume Assumptions \ref{hyp:sigma},
\ref{hyp:cr}, \ref{hyp:rint} to be in effect. Then for every $F\in H_{\rho,0}\left(\mathbb{R},D(C_{0}^{\ast})'\right)$
(with $C_{0}\coloneqq\interior\curl|_{H_{0}}$) there is a unique
(weak) solution $U\in H_{\rho,0}\left(\mathbb{R},D\left(\interior\curl\right)\right)\cap H_{\rho,0}\left(\mathbb{R},H_{0}\right)$
of 
\[
\text{\ensuremath{\left(\overline{\partial_{0}\sigma+\curl\mu^{-1}\interior\curl}\right)}}\:U=F.
\]
Moreover the solution operator $S:H_{\rho,0}\left(\mathbb{R},D(C_{0}^{\ast})'\right)\to H_{\rho,0}\left(\mathbb{R},D\left(\interior\curl\right)\right)$
is continuous ($\left|\:\cdot\:\right|_{\rho,0,1}$ denotes the norm
of $H_{\rho,0}\left(\mathbb{R},D\left(\interior\curl\right)\right)$
and causal in the sense that
\[
\left|\chi_{_{]-\infty,a]}}SF\right|_{\rho,0,1}\leq C_{1}\:\left|\chi_{_{]-\infty,a]}}F\right|_{\rho,0,-1}
\]
for some positive $C_{1}$ uniformly in $a\in\mathbb{R}$ and $F\in H_{\rho,0}\left(\mathbb{R},D(C_{0}^{\ast})'\right)$
as long as $\rho\in\oi0\infty$ is sufficiently large.\end{thm}\begin{proof}The
result follows from Theorem \ref{thm:Solution} in conjunction with
Proposition \ref{prop:setass}.\end{proof}

\red{
\begin{rem}\label{rem:Aphi}There are two engineering type approaches, which have inspired a number
of mathematical investigations see e.g. \cite{zbMATH06435646,zbMATH06418162,zbMATH06064782,zbMATH04191042}
and the literature quoted there. In engineering lingo they are frequently
referred to (by a slight abuse of language, turning adhoc names of variables
into constant names) as the $A$-$\varphi$ approach and the $T$-$\Omega$
approach, where two variants of a vector potential construction come
into play. Our approach is designed precisely to avoid these constructions,
which are actually adding complexity to an already sufficiently complex
topic. A crucial assumption in the application of these approaches
is that the current density source term $J$ is supposed to be divergence-free\footnote{\red{A divergence-free condition is also imposed in the existence result in \cite{RV2010}. We detail the potential formulation in the discussion here. What is said on the divergence condition, however, also applies to the time-harmonic setting focused on in \cite{RV2010}.}},
which, apart from requiring additional regularity of $J$,
excludes perfectly reasonable current densities, say $J=Ie_{3}$,
if $I$ is not completely constant in direction $e_{3}$. In contrast,
we are here considering the eddy current problem directly by solving
\[
\partial_{0}\sigma e+\curl\mu^{-1}\curl e=-J
\]
with a general square-integrable right-hand side (with an exponential
weight in the time direction) with only the obvious constraint that
$J$ is required to be in the closure of the range of $\partial_{0}\sigma+\curl\mu^{-1}\curl$. We emphasise that the present approach allows to recover the original unknowns, see Remark \ref{rem:recovery}.
\end{rem}}
\section{\label{sec:An-Extended-System}An Extended System Formulation for
the Pre-Maxwell System.}

For numerical purposes the construction of $H_{0}$ is not particularly
comfortable. We therefore want to propose an alternative formulation
in the spirit of the extended Maxwell system \cite{key-13,Tas07,PTW_Grav},
which in the context of numerical investigations is of so-called saddle-point
form. In fact, the key is to formulate belonging to $H_{0}^{\bot}$
with the help of belonging to the kernel of certain differential operators. \textcolor{black}{We therefore hope that the proposed reformulation might shed some light on possible numerical implementations of the considered model.}
Quite recently, this method has been applied to homogenisation problems,
see \cite{key-14}. 

Throughout this section, we assume $\Omega$ to be open and bounded
with connected boundary. Moreover, let the Assumptions \ref{hyp:sigma},
\ref{hyp:cr}, \ref{hyp:rint} be in effect. Moreover, we shall rather
focus on $\mu=1$. We need to impose an additional assumption for
this section:

\begin{hypo}\label{hyp:rome} Assume that 
\[
D\left(\interior\grad\right)=\left\{ \psi\in D\left(\grad_{\mathbb{R}^{3}}\right)|\psi=0\text{ on }\mathbb{R}^{3}\setminus\overline{\Omega}\right\} 
\]
as well as 
\[
D(\interior\grad_{\Omega\setminus\overline{\Omega_{c}}})=\{\psi\in D(\interior\grad)\,|\,\psi=0\text{ on }\Omega_{c}\}.
\]

\end{hypo}

\begin{rem} The latter assumption holds for instance, if $\Omega$
and $\Omega_{c}$ satisfy the segment property. \end{rem}

Amending the system in question by an equation in $H_{0}^{\perp_{L^{2}(\Omega)}}$
suitably leads to
\[
\left(\begin{array}{cc}
\left(\begin{array}{cc}
\partial_{0}\sigma+\curl\interior\curl & 0\\
0 & 0
\end{array}\right) & \left(\begin{array}{c}
0\\
\overset{\diamond}{\grad}_{\Omega\setminus\overline{\Omega_{c}}}
\end{array}\right)\\
\left(\begin{array}{cc}
0 & \overset{\diamond}{\dive}_{\Omega\setminus\overline{\Omega_{c}}}\end{array}\right) & 0
\end{array}\right)
\]
with $\left(H_{0}\oplus H_{0}^{\perp_{L^{2}(\Omega)}}\right)\oplus L^{2}(\Omega\setminus\overline{\Omega_{c}},\mathbb{R})$
as underlying Hilbert space. Here we have 
\begin{align*}
\overset{\diamond}{\grad}_{\Omega\setminus\overline{\Omega_{c}}}:D(\overset{\diamond}{\grad}_{\Omega\setminus\overline{\Omega_{c}}})\subseteq L^{2}(\Omega\setminus\overline{\Omega_{c}},\mathbb{R}) & \to H_{0}^{\bot_{L^{2}(\Omega)}}\\
\varphi & \mapsto\interior\grad\varphi
\end{align*}
with
\[
D\left(\overset{\diamond}{\grad}_{\Omega\setminus\overline{\Omega_{c}}}\right)=\left\{ \chi_{_{\Omega\setminus\overline{\Omega_{c}}}}\varphi\,\big|\,\varphi\in D\left(\interior\grad\right),\varphi\text{ constant on }\Omega_{c}\right\} .
\]
To fit our scheme we have let here
\[
\overset{\diamond}{\dive}_{\Omega\setminus\overline{\Omega_{c}}}\coloneqq-\overset{\diamond}{\grad}_{\Omega\setminus\overline{\Omega_{c}}}^{*}.
\]
A reason for the introduction of these new operators is the following
lemma. 

\begin{lem}It is
\[
\mathcal{H}{}_{D,\Omega\setminus\overline{\Omega_{c}}}=R\left(\overset{\diamond}{\grad}_{\Omega\setminus\overline{\Omega_{c}}}\right)\ominus R\left(\interior\grad_{\Omega\setminus\overline{\Omega_{c}}}\right).
\]
\end{lem}\begin{proof}Let
\[
\Phi\in\mathcal{H}{}_{D,\Omega\setminus\overline{\Omega_{c}}}\subseteq N(\interior\curl_{\Omega\setminus\overline{\Omega_{c}}})
\]
and by extension by zero $\Phi\in N\left(\curl_{\mathbb{R}^{3}}\right)$.
Thus 
\[
\Phi=\grad\psi
\]
in $L^{2,\mathrm{loc}}\left(\mathbb{R}^{3},\mathbb{R}^{3}\right)$
for some weakly differentiable $\psi$. Since $\Phi=0$ on $\mathbb{R}^{3}\setminus\left(\Omega\setminus\overline{\Omega_{c}}\right)$
we have that $\psi$ is constant on each component of $\mathbb{R}^{3}\setminus\left(\Omega\setminus\overline{\Omega_{c}}\right)$.
Adjusting this constant to be zero on the unbounded part $\mathbb{R}^{3}\setminus\Omega$
of $\mathbb{R}^{3}\setminus\left(\Omega\setminus\overline{\Omega_{c}}\right)$
we get a $\hat{\psi}\in D\left(\grad_{\mathbb{R}^{3}}\right)$ with
$\hat{\psi}$ constant on $\Omega_{c}$, $\hat{\psi}=0$ on $\mathbb{R}^{3}\setminus\overline{\Omega}$
and 
\[
\Phi=\grad_{\mathbb{R}^{3}}\hat{\psi}.
\]
By Assumption \ref{hyp:rome} we know that
\[
D\left(\interior\grad\right)=\left\{ \psi\in D\left(\grad_{\mathbb{R}^{3}}\right)\,|\,\psi=0\text{ on }\mathbb{R}^{3}\setminus\overline{\Omega}\right\} .
\]
Thus, 
\[
\Phi=\overset{\diamond}{\grad}_{\Omega\setminus\overline{\Omega_{c}}}\hat{\psi}.
\]
Since also $\dive_{\Omega\setminus\overline{\Omega_{c}}}\Phi=0$ we
have indeed shown that 
\begin{eqnarray*}
\mathcal{H}{}_{D,\Omega\setminus\overline{\Omega_{c}}} & \subseteq & R\left(\overset{\diamond}{\grad}_{\Omega\setminus\overline{\Omega_{c}}}\right)\cap N(\dive_{\Omega\setminus\overline{\Omega}_{c}})\\
 & = & R\left(\overset{\diamond}{\grad}_{\Omega\setminus\overline{\Omega_{c}}}\right)\cap R\left(\interior\grad_{\Omega\setminus\overline{\Omega_{c}}}\right)^{\bot_{L^{2}(\Omega\setminus\overline{\Omega_{c}})}}=R\left(\overset{\diamond}{\grad}_{\Omega\setminus\overline{\Omega_{c}}}\right)\ominus R\left(\interior\grad_{\Omega\setminus\overline{\Omega_{c}}}\right).
\end{eqnarray*}
Let now $\Phi=\overset{\diamond}{\grad}_{\Omega\setminus\overline{\Omega_{c}}}\hat{\psi}$
for some $\hat{\psi}\in D(\overset{\diamond}{\grad}_{\Omega\setminus\overline{\Omega_{c}}})$
and $\dive_{\Omega\setminus\overline{\Omega_{c}}}\Phi=0$. Let $\psi_{0}\in D(\interior\grad)$
an extension of $\hat{\psi}|_{\Omega_{c}}$ such that $\psi_{0}$
is constant in a neighbourhood of $\Omega_{c}$. Then, in particular,
$\hat{\psi}-\psi_{0}$ vanishes on $\Omega_{c}$, and so by Assumption
\ref{hyp:rome}
\[
\hat{\psi}-\psi_{0}\in D\left(\interior\grad_{\Omega\setminus\overline{\Omega_{c}}}\right).
\]
We have by construction that
\[
\dive_{\Omega\setminus\overline{\Omega_{c}}}\overset{\diamond}{\grad}_{\Omega\setminus\overline{\Omega_{c}}}\hat{\psi}=\dive_{\Omega\setminus\overline{\Omega_{c}}}\Phi=0
\]
and so
\[
\dive_{\Omega\setminus\overline{\Omega_{c}}}\interior\grad_{\Omega\setminus\overline{\Omega_{c}}}\left(\hat{\psi}-\psi_{0}\right)=-\dive_{\Omega\setminus\overline{\Omega_{c}}}\overset{\diamond}{\grad}_{\Omega\setminus\overline{\Omega_{c}}}\psi_{0}.
\]
Next, we first note that
\[
\interior\grad_{\Omega\setminus\overline{\Omega_{c}}}\left(\hat{\psi}-\psi_{0}\right)\in N\left(\interior\curl_{\Omega\setminus\overline{\Omega_{c}}}\right).
\]
Since also $\overset{\diamond}{\grad}_{\Omega\setminus\overline{\Omega_{c}}}\psi_{0}\in N\left(\curl_{\Omega\setminus\overline{\Omega_{c}}}\right)\cap N(\interior\curl)$
and since $\overset{\diamond}{\grad}_{\Omega\setminus\overline{\Omega_{c}}}\psi_{0}$
actually vanishes in a neighbourhood of $\Omega_{c}$ we also have
\[
\overset{\diamond}{\grad}_{\Omega\setminus\overline{\Omega_{c}}}\psi_{0}\in N\left(\interior\curl_{\Omega\setminus\overline{\Omega_{c}}}\right).
\]
Thus, 
\[
\Phi=\overset{\diamond}{\grad}_{\Omega\setminus\overline{\Omega_{c}}}\left(\hat{\psi}\right)=\interior\grad_{\Omega\setminus\overline{\Omega_{c}}}\left(\hat{\psi}-\psi_{0}\right)+\overset{\diamond}{\grad}_{\Omega\setminus\overline{\Omega_{c}}}\psi_{0}\in N\left(\interior\curl_{\Omega\setminus\overline{\Omega_{c}}}\right)
\]
and so 
\[
\Phi\in\mathcal{H}{}_{D,\Omega\setminus\overline{\Omega_{c}}}.
\]
This yields the converse inclusion.\end{proof}

 The latter lemma particularly implies 
\begin{eqnarray*}
H_{0}^{\perp_{L^{2}(\Omega)}} & = & N\left(\interior\curl_{\Omega\setminus\overline{\Omega_{c}}}\right)\\
 & = & R\left(\interior\grad_{\Omega\setminus\overline{\Omega_{c}}}\right)\oplus\mathcal{H}{}_{D,\Omega\setminus\overline{\Omega_{c}}}\\
 & = & R\left(\overset{\diamond}{\grad}_{\Omega\setminus\overline{\Omega_{c}}}\right),
\end{eqnarray*}
where we have used Lemma \ref{lem:decom} for the first equality.
Since, according to the projection theorem, the canonical embedding
\begin{eqnarray*}
\left(\begin{array}{cc}
\iota_{H_{0}} & \iota_{H_{0}^{\perp}}\end{array}\right):H_{0}\oplus H_{0}^{\perp_{L^{2}(\Omega)}} & \to & L^{2}(\Omega,\mathbb{R}^{3})\\
\left(\begin{array}{c}
x_{0}\\
x_{1}
\end{array}\right) & \mapsto & x_{0}+x_{1}
\end{eqnarray*}
is unitary we have its adjoint
\[
\left(\begin{array}{c}
\iota_{H_{0}}^{*}\\
\iota_{H_{0}^{\perp}}^{*}
\end{array}\right):L^{2}\left(\Omega,\mathbb{R}^{3}\right)\to H_{0}\oplus H_{0}^{\perp_{L^{2}(\Omega)}}
\]
as the inverse. Thus, we may consider equivalently
\begin{eqnarray*}
W\left(\begin{array}{cc}
\left(\begin{array}{cc}
\partial_{0}\sigma+\curl\interior\curl & 0\\
0 & 0
\end{array}\right) & \left(\begin{array}{c}
0\\
\overset{\diamond}{\grad}_{\Omega\setminus\overline{\Omega_{c}}}
\end{array}\right)\\
\left(\begin{array}{cc}
0 & \overset{\diamond}{\dive}_{\Omega\setminus\overline{\Omega_{c}}}\end{array}\right) & 0
\end{array}\right)W^{*}=\\
=\left(\begin{array}{cc}
\partial_{0}\sigma+\curl\interior\curl & \overset{\diamond}{\grad}_{\Omega\setminus\overline{\Omega_{c}}}\\
\overset{\diamond}{\dive}_{\Omega\setminus\overline{\Omega_{c}}} & 0
\end{array}\right)
\end{eqnarray*}
now on $L^{2}\left(\Omega,\mathbb{R}^{3}\right)\oplus L^{2}\left(\Omega\setminus\overline{\Omega_{c}},\mathbb{R}\right)$
as underlying Hilbert space with the unitary map
\[
W=\left(\begin{array}{cc}
\left(\begin{array}{cc}
\iota_{H_{0}} & \iota_{H_{0}^{\perp}}\end{array}\right) & 0\\
\left(\begin{array}{cc}
0_{H_{0}} & 0_{H_{0}^{\perp}}\end{array}\right) & 1
\end{array}\right).
\]
Thus, we are led to discuss equations of the form
\[
\left(\begin{array}{cc}
\partial_{0}\sigma+\curl\interior\curl & \overset{\diamond}{\grad}_{\Omega\setminus\overline{\Omega_{c}}}\\
\overset{\diamond}{\dive}_{\Omega\setminus\overline{\Omega_{c}}} & 0
\end{array}\right)\left(\begin{array}{c}
E\\
p
\end{array}\right)=\left(\begin{array}{c}
f\\
0
\end{array}\right).
\]
From this ``saddle point formulation'' we can recover $E$ as the
solution of
\begin{equation}
\partial_{0}\sigma E+\curl\interior\curl E=\iota_{H_{0}}^{\ast}f.\label{eq:pre-Max-1}
\end{equation}
Indeed, we have the following result.

\begin{thm} Assume $\Omega$ to be open and bounded with connected
boundary. Moreover, let the Assumptions \ref{hyp:sigma}, \ref{hyp:cr},
\ref{hyp:rint}, and \ref{hyp:rome} be in effect. Then the (closure
of the) operator 
\[
\left(\begin{array}{cc}
\left(\begin{array}{cc}
\partial_{0}\sigma+\curl\interior\curl & 0\\
0 & 0
\end{array}\right) & \left(\begin{array}{c}
0\\
\overset{\diamond}{\grad}_{\Omega\setminus\overline{\Omega_{c}}}
\end{array}\right)\\
\left(\begin{array}{cc}
0 & \overset{\diamond}{\dive}_{\Omega\setminus\overline{\Omega_{c}}}\end{array}\right) & 0
\end{array}\right)
\]
 is continuously invertible in $H_{\rho,0}(\mathbb{R},H_{0}\oplus H_{0}^{\bot}\oplus L^{2}\left(\Omega\setminus\overline{\Omega_{c}},\mathbb{R}\right))$
for sufficiently large $\rho>0$.

\end{thm}

\begin{proof} Note that since $\Omega$ is open and bounded, we infer
by Poincare's inequality that $R\left(\interior\grad\right)$ is closed.
This implies that $R\left(\overset{\diamond}{\grad}_{\Omega\setminus\overline{\Omega_{c}}}\right)$
is closed as well as the range $R\left(\overset{\diamond}{\dive}_{\Omega\setminus\overline{\Omega_{c}}}\right)$
of its adjoint $-\stackrel{\diamond}{\dive}_{\Omega\setminus\overline{\Omega_{c}}}$.
This makes 
\[
\left(\begin{array}{cc}
0 & \overset{\diamond}{\grad}_{\Omega\setminus\overline{\Omega_{c}}}\\
\overset{\diamond}{\dive}_{\Omega\setminus\overline{\Omega_{c}}} & 0
\end{array}\right)
\]
continuously invertible on $R\left(\overset{\diamond}{\grad}_{\Omega\setminus\overline{\Omega_{c}}}\right)\oplus R\left(\overset{\diamond}{\dive}_{\Omega\setminus\overline{\Omega_{c}}}\right)$.
Moreover, it is a consequence of the above lemma that
\[
R\left(\overset{\diamond}{\grad}_{\Omega\setminus\overline{\Omega_{c}}}\right)=H_{0}^{\perp_{L^{2}(\Omega)}}.
\]
Furthermore, since $\interior\grad$ is injective, we infer that
\[
N\left(\overset{\diamond}{\grad}_{\Omega\setminus\overline{\Omega_{c}}}\right)=\{0\},
\]
which, thus, implies that 
\[
R\left(\overset{\diamond}{\dive}_{\Omega\setminus\overline{\Omega_{c}}}\right)=L^{2}\left(\Omega\setminus\overline{\Omega_{c}},\mathbb{R}\right).
\]
Hence, we infer the claim of the theorem by the well-posedness result
from Theorem \ref{thm:mr}.\end{proof}

The solution $\left(E,p\right)$ of the extended system now yields
indeed a solution $E$ of the pre-Maxwell system \eqref{eq:pre-Max-1}.
If $f\in H_{0}$ we have of course $f=P_{H}{}_{0}f$ and $p=0$.

\begin{rem}For numerical purposes approximations of the equation
$\overset{\diamond}{\dive}_{\Omega\setminus\overline{\Omega_{c}}}E=0$
would be based on its 'weak' form 
\[
\left\langle \overset{\diamond}{\grad}_{\Omega\setminus\overline{\Omega_{c}}}\psi|E\right\rangle _{L^{2}\left(\Omega\setminus\overline{\Omega_{c}},\mathbb{R}^{3}\right)}=0,
\]
so that $E$ could be approximated in suitable finite-dimensional
subspaces of $D\left(\interior\curl\right)$.\end{rem}

\section{Justification of the Pre-Maxwell System.}

We conclude our considerations with a justification of the pre-Maxwell
system; that is, the degenerate eddy current problem\footnote{For the non-degenerate eddy current problem this has been given in
the current functional analytical setting in \cite{MMA:MMA4515,Wau16}
in both the autonomous and non-autonomous cases, respectively.}, as an approximation of Maxwell's system (including the displacement
current). The system of Maxwell's equations reads as
\begin{eqnarray*}
\partial_{0}\epsilon\mathrm{E}+\sigma\mathrm{E}-\curl\mathrm{H} & = & -\mathrm{J,}\\
\partial_{0}\mu\mathrm{H}+\interior\curl\mathrm{E} & = & \mathrm{K},
\end{eqnarray*}
where $K$ denotes a magnetic source term (perhaps induced by initial
data for $H$) and $\epsilon\in]0,\infty[$. Throughout, let $\rho\geq1$.
The question is if and in which sense do the solutions converge to
the solutions of the degenerate eddy current problem as $\epsilon$
tends to $0$. For this transition we restrict our attention to current
densities $J$ in the correct subspace for the limit problem $\epsilon=0$,
i.e.
\[
J\in H_{\rho,0}(\mathbb{R},H_{0}).
\]

Again, as before, we shall assume that $\Omega$ is open, bounded
with connected boundary. Furthermore, we shall assume throughout that
the Assumptions \ref{hyp:sigma}, \ref{hyp:cr}, \ref{hyp:rint} are
in effect. We shall furthermore note that a standard application of
Theorem \ref{-Solution-0} leads to 
\[
\tilde{S}_{\varepsilon}\coloneqq\left(\overline{\left(\partial_{0}\left(\begin{array}{cc}
\epsilon & 0\\
0 & \mu
\end{array}\right)+\left(\begin{array}{cc}
\sigma & 0\\
0 & 0
\end{array}\right)+\left(\begin{array}{cc}
0 & -\curl\\
\interior\curl & 0
\end{array}\right)\right)}\right)^{-1}\in L(H_{\rho,k}(\mathbb{R};L^{2}(\Omega,\mathbb{R}^{6})))
\]
for \emph{every} $\rho>0$ and $k\in\mathbb{Z}$. Here and in the
following we use $\left|\:\cdot\:\right|_{\rho,k,0}$ as the notation
for the norm corresponding to the Hilbert space inner product induced
by $\left\langle \:\cdot\:|\:\cdot\:\right\rangle _{\rho,k,0}\coloneqq\left\langle \partial_{0}^{k}\:\cdot\:|\partial_{0}^{k}\:\cdot\:\right\rangle _{\rho,0,0}$.
$H_{\rho,k}\left(\mathbb{R},L^{2}(\Omega,\mathbb{R}^{6})\right)$
denotes the Hilbert space obtained by completion. We denote 
\[
S_{0}\coloneqq\left(\overline{\partial_{0}\sigma+\curl\mu^{-1}\interior\curl}\right)^{-1}\in L(H_{\rho,0}(\mathbb{R},H_{0}),H_{\rho,0}(\mathbb{R},D(\interior\curl))
\]
 for some fixed sufficiently large $\rho>0$. Furthermore, we define
for all $\varepsilon>0$
\[
S_{\varepsilon}\coloneqq\pi_{1}\tilde{S}_{\varepsilon},
\]
where $\pi_{1}(\mathrm{E},\mathrm{H})=\mathrm{E}$ reads off the first
three components of a 6-component vector field. Assuming 
\[
\curl\mu^{-1}\partial_{0}^{-1}\mathrm{K}\in H_{\rho,k}\left(\mathbb{R},L^{2}\left(\Omega,\mathbb{R}^{3}\right)\right)
\]
the simple substitution 
\[
\mathrm{H}=\mu^{-1}\partial_{0}^{-1}\mathrm{K}-\mu^{-1}\interior\curl\partial_{0}^{-1}\mathrm{E}
\]
leads to $S_{\varepsilon}(\mathrm{J},\mathrm{K})=\mathrm{E}$ being
the unique solution of 
\[
\partial_{0}\epsilon\mathrm{E}+\sigma\mathrm{E}+\curl\mu^{-1}\interior\curl\partial_{0}^{-1}\mathrm{E}=-\mathrm{J}+\curl\mu^{-1}\partial_{0}^{-1}\mathrm{K}.
\]
By a slight abuse of notation, we shall view $S_{\epsilon}$ as a
mapping from $H_{\rho,0}(\mathbb{R};L^{2}(\Omega,\mathbb{R}^{3}))$
into itself. Thus, instead of $S_{\varepsilon}(\mathrm{J},\mathrm{K})$
we shall write $S_{\varepsilon}(-\mathrm{J}+\curl\mu^{-1}\partial_{0}^{-1}\mathrm{K})$.
This provides a second order formulation, which we actually can compare
with the degenerate equation. Due to the particular structure of the
right-hand side, we furthermore remark here that $f=-\mathrm{J}+\curl\mu^{-1}\partial_{0}^{-1}\mathrm{K}$
takes values in $H_{0}$ if and only if $J$ does. The main result
of this section reads as follows.

\begin{thm}\label{thm:mrjust} For all $k\in\mathbb{Z}$ and $f\in H_{\rho,k}(\mathbb{R};H_{0})$
we have
\[
\left|S_{\epsilon}f-S_{0}f\right|_{\rho,k-2,0}\to0
\]
as $\varepsilon\to0$.

\end{thm}

Before proving this result, we provide the following an auxiliary
result.

\begin{lem}\label{lem:mrjust} For all $k\in\mathbb{Z}$ , we have
\[
\sup_{\varepsilon>0}\|S_{\varepsilon}\|_{H_{\rho,k}(\mathbb{R},H_{0})\to H_{\rho,k-2}(\mathbb{R},H_{0})}<\infty.
\]

\end{lem}

\begin{proof} Let $f\in H_{\rho,k+1}(\mathbb{R},H_{0})$, $\varepsilon>0$.
Then $\mathrm{E}=S_{\varepsilon}f$ satisfies
\[
\partial_{0}\epsilon\mathrm{E}+\sigma\mathrm{E}+\curl\mu^{-1}\interior\curl\partial_{0}^{-1}\mathrm{E}=f.
\]

We shall now separate this equation into the parts in $H_{0}$ and
$H_{0}^{\perp}$ separately. 

Denoting $\left(\begin{array}{c}
\mathrm{E}_{0}\\
\mathrm{E}_{1}
\end{array}\right)=\left(\begin{array}{c}
\iota_{H_{0}}^{*}\mathrm{E}\\
\iota_{H_{0}^{\perp}}^{*}\mathrm{E}
\end{array}\right),$ we obtain 
\begin{align*}
\partial_{0}\varepsilon E_{0}+\sigma E_{0}+\curl\mu^{-1}\interior\curl\partial_{0}^{-1}E_{0} & =\iota_{H_{0}}^{\ast}f,\\
\partial_{0}\varepsilon E_{1} & =0,
\end{align*}
where we have used that $f\in H_{0}.$ By the second equation we have
\[
\partial_{0}\varepsilon\mathrm{E}_{1}=0
\]
and thus, continuous invertibility of $\partial_{0}$ implies $\mathrm{E}_{1}=0$.
Testing the equation for $\mathrm{E}_{0}$ with $\mathrm{E}_{0}$,
we deduce
\[
\rho\left|\epsilon^{1/2}\mathrm{E}_{0}\right|_{\rho,k,0}^{2}+\left|\sigma^{1/2}\mathrm{E}_{0}\right|_{\rho,k,0}^{2}+\left\langle \interior\curl\mathrm{E}_{0}|\mu^{-1}\partial_{0}^{-1}\interior\curl\mathrm{E}_{0}\right\rangle _{\rho,k,0}=\left\langle \mathrm{E}_{0}|f\right\rangle _{\rho,k,0}\leq\left|\mathrm{E}_{0}\right|_{\rho,k-1,0}\left|f\right|_{\rho,k+1,0}.
\]
Using
\[
\left|\sigma^{1/2}\mathrm{E}_{0}\right|_{\rho,k-1,0}\leq\frac{1}{\rho}\left|\sigma^{1/2}\mathrm{E}_{0}\right|_{\rho,k,0},
\]
and 
\begin{align*}
\left\langle \interior\curl\mathrm{E}_{0}|\mu^{-1}\partial_{0}^{-1}\interior\curl\mathrm{E}_{0}\right\rangle _{\rho,k,0} & =\left\langle \partial_{0}\partial_{0}^{-1}\interior\curl\mathrm{E}_{0}|\mu^{-1}\partial_{0}^{-1}\interior\curl\mathrm{E}_{0}\right\rangle _{\rho,k,0}\\
 & =\rho\left|\partial_{0}^{-1}\mu^{-1/2}\interior\curl\mathrm{E}_{0}\right|_{\rho,k,0}^{2}\\
 & =\rho\left|\mu^{-1/2}\interior\curl\mathrm{E}_{0}\right|_{\rho,k-1,0}^{2}
\end{align*}
we infer
\[
\rho^{2}\left|\sigma^{1/2}\mathrm{E}_{0}\right|_{\rho,k-1,0}^{2}+\rho\left|\mu^{-1/2}\interior\curl\mathrm{E}_{0}\right|_{\rho,k-1,0}^{2}\leq|E_{0}|_{\rho,k-1,0}|f|_{\rho,k+1,0}.
\]
On the other hand we know by \eqref{eq:pos-def-eddy} that
\[
\left|\sigma^{1/2}\mathrm{E}_{0}\right|_{\rho,k-1,0}^{2}+\left|\mu^{-1/2}\interior\curl\mathrm{E}_{0}\right|_{\rho,k-1,0}^{2}\geq c_{0}\left|\mathrm{E}_{0}\right|_{\rho,k-1,0}^{2}
\]
 for some $c_{0}\in\oi0\infty$. Thus, as $\rho\geq1$ we have
\begin{align*}
c_{0}\left|\mathrm{E}_{0}\right|_{\rho,k-1,0}^{2} & \leq\left|\mathrm{E}_{0}\right|_{\rho,k-1,0}\left|f\right|_{\rho,k+1,0}.
\end{align*}
Consequently, we have the uniform estimate
\[
c_{0}\left|\mathrm{E}_{0}\right|_{\rho,k-1,0}\leq\left|f\right|_{\rho,k+1,0},
\]
which yields 
\[
\sup_{\varepsilon>0}\|S_{\varepsilon}\|_{H_{\rho,k}(\mathbb{R},H_{0})\to H_{\rho,k-2}(\mathbb{R},H_{0})}=\sup_{\varepsilon>0}\|S_{\varepsilon}\|_{H_{\rho,k+1}(\mathbb{R},H_{0})\to H_{\rho,k-1}(\mathbb{R},H_{0})}\leq\frac{1}{c_{0}}.\tag*{\qedhere}
\]
\end{proof}

\begin{proof}[Proof of Theorem \ref{thm:mrjust}]For $\varepsilon>0$
and $f\in H_{\rho,k+1}(\mathbb{R},H_{0})$ we find
\begin{eqnarray*}
S_{\epsilon}f-S_{0}f & = & S_{\epsilon}\left(S_{0}^{-1}-S_{\epsilon}^{-1}\right)S_{0}f\\
 & = & S_{\epsilon}\epsilon\partial_{0}S_{0}f\\
 & = & S_{\epsilon}\epsilon\partial_{0}S_{0}f
\end{eqnarray*}
and so
\begin{eqnarray}\label{eq:qe}
\left|S_{\epsilon}f-S_{0}f\right|_{\rho,k-2,0} & = & \left|S_{\epsilon}\epsilon\partial_{0}S_{0}f\right|_{\rho,k-2,0}\\ \notag
 & \leq & \|S_{\varepsilon}\|_{H_{\rho,k}(\mathbb{R},H_{0})\to H_{\rho,k-2}(\mathbb{R},H_{0})}\left|\epsilon\partial_{0}S_{0}f\right|_{\rho,k,0}\\ \notag
 & \leq & \epsilon\|S_{\varepsilon}\|_{H_{\rho,k}(\mathbb{R},H_{0})\to H_{\rho,k-2}(\mathbb{R},H_{0})}\left|S_{0}\partial_{0}f\right|_{\rho,k,0}\\ \notag
 & \leq & \epsilon\|S_{\varepsilon}\|_{H_{\rho,k}(\mathbb{R},H_{0})\to H_{\rho,k-2}(\mathbb{R},H_{0})}\|S_{0}\|_{H_{\rho,k}(\mathbb{R},H_{0})\to H_{\rho,k}(\mathbb{R},H_{0})}\left|\partial_{0}f\right|_{\rho,k,0}\\ \notag
 & \leq & \epsilon\|S_{\varepsilon}\|_{H_{\rho,k}(\mathbb{R},H_{0})\to H_{\rho,k-2}(\mathbb{R},H_{0})}\|S_{0}\|_{H_{\rho,k-1}(\mathbb{R},H_{0})\to H_{\rho,k-1}(\mathbb{R},H_{0})}\left|f\right|_{\rho,k+1,0}.
\end{eqnarray}
By Lemma \ref{lem:mrjust}, we deduce that 
\[
\left|S_{\epsilon}f-S_{0}f\right|_{\rho,k-2,0}\overset{\varepsilon\to0}{\to}0
\]
for every $f\in H_{\rho,k+1}\left(\mathbb{R},H_{0}\right)$. By density
of $H_{\rho,k+1}\left(\mathbb{R},H_{0}\right)$ in $H_{\rho,k}\left(\mathbb{R},H_{0}\right)$
and uniform boundedness of $\left(S_{\epsilon}\right)_{\epsilon\geq0}$
it follows that
\[
\left|S_{\epsilon}f-S_{0}f\right|_{\rho,k-2,0}\overset{\varepsilon\to0}{\to}0
\]
for all $f\in H_{\rho,k}\left(\mathbb{R},H_{0}\right)$, which is
the desired convergence result. \end{proof}

\red{\begin{rem}\label{rem:Conv} The justification of the eddy-current model as the low electric permittivity limit of the classical Maxwell system is performed in \cite[Theorem 2.5]{RV2010} with a focus on the frequency domain for fixed frequency. The quantitative estimate is of the order $O(\varepsilon)$ as $\varepsilon\to 0$. The estimates and derivations described in the proof above provide the same quantitative nature for fixed frequency (see the estimates in \eqref{eq:qe}). Since the above result covers the full time line (and thus all frequencies) \emph{simultaneously} some time regularity loss has to be expected if one wants to keep the order of $\varepsilon$ result. Indeed, also in the \cite[proof of Theorem 2.5]{RV2010} the frequency dependence of the quantitative estimate suggests a (time) regularity loss if one wants to keep the derived quantitative estimate for the full space-time problem (note the $\omega^2$ in the \cite[proof of Theorem 2.5]{RV2010}). Furthermore, this effect has been observed in the context of Maxwell's equations in \cite{MMA:MMA4515,Wau16}. A similar observation can be made for approximations in quantitative homogenisation theory: Whereas for fixed frequencies one obtains optimal quantitative estimates \cite{CW17_1D,CW17_FH}, the estimates for the full space-time problem experience a loss of derivatives if one wants to retain the same quantitative behaviour, see \cite{W18_ONC_PAMM,FW17_1D}. It is possible to accommodate for this regularity loss with an analogue of Littlewood-Paley type spaces, see \cite{CW17_1D}.
\end{rem}}

\bibliographystyle{abbrv}

\restoregeometry
\end{document}